\documentclass[11pt,reqno]{amsart}
\usepackage{amscd,amsmath,amsopn,amssymb,amsthm,multicol,multirow,xcolor}
\usepackage[normalem]{ulem} 

\usepackage{tikz-cd,subdepth,anysize}
\usetikzlibrary{decorations.markings,arrows.meta,bending,calc}
\tikzset{>={Stealth[scale=1.5, bend]}}
\everymath=\expandafter{\the\everymath\displaystyle}
\usetikzlibrary{positioning}

\usetikzlibrary{arrows,decorations.markings,positioning}
\usetikzlibrary{matrix,decorations.pathmorphing}
\tikzset{inner sep=0pt, node distance=5mm,
  root/.style={circle,draw,minimum size=5pt,thick},
  broot/.style={circle,draw,minimum size=5pt,thick,fill},
  xroot/.style={circle,draw,minimum size=5pt,thick,label=below:$\times$},
  doublearrow/.style={postaction={decorate},   decoration={markings,mark=at position .6 with {\arrow[line width=1.2pt]{>}}},double distance=1.6pt,thick},
  rdoublearrow/.style={postaction={decorate},   decoration={markings,mark=at position .4 with {\arrowreversed[line width=1.2pt]{>}}},double distance=1.6pt,thick},
	rtriplearrow/.style={postaction={decorate},   decoration={markings,mark=at position .4 with {\arrowreversed[line width=1.2pt]{>}}},double distance=2.5pt,thick},
	ltriplearrow/.style={postaction={decorate},   decoration={markings,mark=at position .6 with {\arrow[line width=1.2pt]{>}}},double distance=2.5pt,thick},
  curvedline/.style={bend=right}
}

\newcommand\com[1]{}
\newcommand\C{{\mathbb C}}
\newcommand\D{{\mathcal D}}
\newcommand\E{\mathcal{E}}
\newcommand\g{{\mathfrak g}}
\newcommand\op[1]{\mathop{\rm #1}\nolimits}
\newcommand\p{\partial}
\newcommand\R{{\mathbb R}}
\newcommand\Sing{\mathcal{S}}

\theoremstyle{plain}
\newtheorem{theorem}{Theorem}
\newtheorem{prop}{Proposition}
\newtheorem{cor}{Corollary}
\newtheorem{lemma}{Lemma}
\theoremstyle{definition}
\newtheorem{definition}{Definition}

\newtheorem{remark}{Remark}

\begin{document}

\title[ODEs whose symmetry groups are not fiber-preserving]{ODEs whose symmetry groups are not fiber-preserving}

\author[Boris Kruglikov]{Boris Kruglikov$^\dagger$}
\address{$^\dagger{}^\ddagger$ Department of Mathematics and Statistics, UiT The Arctic University of Norway, Troms\o\ 9037, Norway.}

\author[Eivind Schneider]{Eivind Schneider$^\ddagger$}
\address{$^\ddagger$ Faculty of Science, University of Hradec Kr\'alov\'e, Rokitansk\'eho 62, Hradec Kr\'alov\'e 50003, Czech Republic.}

\address{Email addresses:\qquad {\tt boris.kruglikov@uit.no}\quad\text{\rm and }\quad {\tt eivind.schneider@uit.no}\hspace{1pt}.}


\begin{abstract}
We observe that, up to conjugation, a majority of symmetric higher order ODEs (ordinary differential equations)
and ODE systems have only fiber-preserving point symmetries.
By exploiting  Lie's classification of Lie algebras of vector fields, we describe all the exceptions to this in the case
of scalar ODEs and systems of ODEs on a pair of functions.

The scalar ODEs whose symmetry algebra is not fiber preserving can be expressed via absolute and
relative scalar differential invariants, while a similar description for ODE systems requires us to also invoke
conditional differential invariants and vector-valued 
relative invariants to deal with singular orbits of the action.

Investigating prolongations of the actions, we observe some interesting relations between
different realizations of Lie algebras. We also note that it may happen that the prolongation of a finite-dimensional Lie algebra acting on a differential
equation never becomes free. An example of an underdetermined ODE system for which this phenomenon occurs shows limitations of the method of moving frames.
\end{abstract}

\maketitle


\tableofcontents

\section*{Introduction}

It is well known that among the most symmetric {\em scalar\/} ODEs
of order $n>2$ only the equation $y'''=0$ has
an irreducible contact symmetry algebra, which is\footnote{We work exclusively over $\C$;
some facts remain true over $\R$ but the classification is longer. We use the simplified notation $\mathfrak{sp}(4)=\mathfrak{sp}(4,\C)$, $\mathfrak{gl}(2)=\mathfrak{gl}(2,\C)$, etc.}
$\mathfrak{sp}(4)$. For every $n>3$ the ordinary differential equation
$y^{(n)}=0$ has the symmetry algebra
$\mathfrak{gl}(2)\ltimes S^{n-1}\C^2$ of $\dim=n+4$ with generators
embedded as fiber-preserving vector fields on the fiber bundle
$J^0=\C^2(x,y)\to\C^1(x)$, see for example \cite[Chapt. 6]{O}.

There are indeed ODEs with non-point symmetries. Consider, for instance,
the following two equations (from this moment on we use the jet-notation $y_n$ instead of $y^{(n)}$
for the derivative):
 $$
(i)\quad  y_4=3y_3^2/y_2,\qquad
(ii)\quad y_4=3y_3^2/y_2+y_3^2/y_2^2.
 $$
Then (i) has (three) contact non-point symmetries but is trivializable,
while (ii) has (only one) contact non-point symmetry and is not
even linearizable (this follows from the Lie algebra structure of
its symmetry $\mathfrak{sol}(2)\ltimes\C^3$ because it does not contain
a four-dimensional Abelian subalgebra). However the Legendre
transform $(x,y,y_1)\mapsto(-y_1,y-xy_1,x)$ maps these into resp.\
({\it i}) $y_4=0$, ({\it ii}) $y_4=y_3^2$
which have only fiber-preserving point transformations as contact symmetries.


Let us call a subalgebra of the Lie algebra of contact vector fields on
$J^1=\C^3(x,y,y_1)$ {\em essentially fiber-preserving\/} if it is
conjugate by a contact transformation to (the prolongation of) a subalgebra of point transformations preserving the foliation $\{x=\op{const}\}$.
Similarly, call a subalgebra of the Lie algebra of contact vector fields on $J^1$ {\em essentially point\/} if it is conjugate to (the prolongation of) a subalgebra of vector fields on $J^0=\C^2(x,y)$ (we exclude from those
the essentially fiber-preserving ones). Finally, the rest will be called
{\em essentially contact\/} algebras.

Thus the question: To what extent is it true that higher order ODEs have essentially point and
even essentially fiber-preserving Lie algebras of symmetries?
Clearly $y_3=0$ is an exception. Among ODEs of order $n>3$ with submaximal symmetry dimension
(equal to $n+2$ for $n\neq5,7$ and $n+3$ for $n=5,7$) all have essentially fiber-preserving
point symmetries but with three exceptions for $n=4,5,7$, see
\cite[p.205-206]{O}:
 \begin{gather*}
L_4[y]=3y_2y_4-5y_3^2=0,
L_5[y]=9y_2^2y_5-45y_2y_3y_4+40y_3^3=0,\\
L_7[y]=10y_3^3y_7-70y_3^2y_4y_6-49y_3^2y_5^2+280y_3y_4^2y_5-175y_4^4=0.
 \end{gather*}
 The first two have point symmetry algebras
$\mathfrak{aff}(2)$ and $\mathfrak{sl}(3)$ respectively, while the
last one has $\mathfrak{sp}(4)$ as a contact symmetry algebra.

It turns out that there are many more scalar ODEs of order greater than two with essentially contact and
essentially point symmetry algebras, yet they are minor among all ODEs possessing nontrivial infinitesimal symmetries.

In this paper we describe all those exceptions, basing on the
original ideas of Sophus Lie. Namely we compute
the algebras $\mathcal{R}$ of relative differential invariants
with respect to the smallest irreducible Lie algebra \eqref{L1} of contact vector fields on $J^1$
and with respect to the smallest primitive Lie algebra \eqref{L2} of vector fields on $J^0$.
This yields all algebraic equations with essentially contact and respectively essentially
point symmetry algebras. To cover analytic equations with essentially contact or point symmetries
we describe the field of absolute differential invariants, and we also find the
singular equations (consisting of singular orbits of the action). This is done in \S\ref{sect:Scalar}.
Let us note that the results of this section to a certain extent are known; this is due to
the works of Lie with notable later contributions, see \cite{L1,O,DK,WMQ}. We however make some
specifications of global nature on algebraic equations and the algebra of relative invariants,
which appear to be new. We will make some historical remarks in the Conclusion section.

In the case of ODE {\em systems\/} the situation is more complicated. We consider in detail the case of pairs of ODEs. Here the list of possible Lie algebras is larger, and the
tools of scalar relative and absolute invariants is no longer sufficient. Indeed, any such ODE system
is given by two differential constraints, possibly of different orders, and invariance of the locus
does not imply that  each of the two corresponding differential equations can be chosen such that they are invariant by themselves.
In \S\ref{sect:Sytems} we first formulate the strategy of how to find the invariant ODE systems.
Then we describe all algebraic and analytic equations with essentially point (not fiber-preserving) symmetry algebras. The essentially point Lie algebras are in this case the Lie algebras of point symmetries that preserve either no foliation or a 1-dimensional foliation but no 2-dimensional foliation.

\begin{table}[h!]
\begin{tabular}{|c|l|l|l|}
\hline
Lie algebra $\mathfrak g$ & Invariant condition & Scalar invariants & $\mathfrak g$-inv. ODE system $\mathcal E$ \\\hline \hline
\multirow{8}{*}{\begin{tabular}{c} $\mathfrak{so}(3) \ltimes \mathbb C^3$ \\ \\ \\ Ref: \S \ref{sect:Minkowski}\\ Th. \ref{th:InvarMinkowski}, \ref{th:invarEqMinkowski} \end{tabular}} &  & $\mathcal A = \langle I_2, I_{3a}, \nabla\rangle$ & \begin{tabular}{c} \rule[-0ex]{0pt}{2.1ex} $\{F=0, G=0\}$ \\ $F, G \in \mathcal A$\end{tabular}  \\\cline{2-4}
 & \multirow{2}{*}{$\Sigma\colon R_1=0$}  & $\mathcal A_{\Sigma}=\langle  J_4, \nabla_\Sigma\rangle$ & \begin{tabular}{c} \rule[-0ex]{0pt}{2.1ex}$\{F=0, R_1=0\}$ \\ $F \in \mathcal A_{\Sigma}$ \end{tabular}  \\\cline{3-4}
  & & $Q_2$ & \rule[-0ex]{0pt}{2.1ex}$\{Q_2=0, R_1=0\}$  \\\cline{2-4}
   & $\Pi \colon R_2=0$ & $\mathcal A_{\Pi}=\langle K_3, \nabla_\Pi\rangle$ & \begin{tabular}{c}\rule[-0ex]{0pt}{2.1ex} $ \{F=0, R_2=0\}$ \\ $F \in \mathcal A_{\Pi}$ \end{tabular}  \\\cline{2-4}
    & $y_2=0, z_2=0$ &  & \rule[-0ex]{0pt}{2.1ex} $ \{y_2=0, z_2=0\}$   \\\hline \hline
\multirow{8}{*}{\begin{tabular}{c} $\mathfrak{so}(4)$ \\ \\  \\ Ref: \S\ref{sect:so4} \\ Th. \ref{th:InvarSpaceform}, \ref{th:invarEqSpaceform} \end{tabular}} &  & $\mathcal A = \langle I_2, I_{3a}, \nabla\rangle$ & \begin{tabular}{c}\rule[-0ex]{0pt}{2.1ex} $ \{F=0, G=0\}$ \\ $F, G \in \mathcal A$\end{tabular}   \\\cline{2-4}
 & \multirow{2}{*}{$\Sigma \colon R_1=0$}  & $\mathcal A_{\Sigma}=\langle  J_4, \nabla_\Sigma\rangle$ & \begin{tabular}{c} \rule[-0ex]{0pt}{2.1ex}$ \{F=0, R_1=0\}$ \\ $F \in \mathcal A_{\Sigma}$ \end{tabular}  \\\cline{3-4}
  & & $Q_2$ & \rule[-0ex]{0pt}{2.1ex} $\{Q_2=0, R_1=0\}$  \\\cline{2-4}
   & $ \Pi \colon R_2=0$ & $\mathcal A_{\Pi}=\langle K_3, \nabla_\Pi \rangle$ & \begin{tabular}{c} \rule[-0ex]{0pt}{2.1ex} $\{F=0, R_2=0\}$ \\ $F \in \mathcal A_{\Pi}$ \end{tabular}  \\\cline{2-4}
    & \rule[-3ex]{0pt}{7ex} \begin{tabular}{l} $y_2=y_1(1+ \tfrac{2y_1 z_1}{e^x})$ \\ $z_2=z_1(1+\tfrac{2y_1 z_1}{e^x})$ \end{tabular} &  & $\left\{\tfrac{y_2}{y_1} = 1+\tfrac{2 y_1 z_1}{e^x} = \tfrac{z_2}{z_1}\right\}$   \\\hline  \hline
    \multirow{10}{*}{\begin{tabular}{c}$\mathfrak{sp}(4)$ \\ \\ \\ Ref: \S\ref{sect:sp4} \\ Th. \ref{th:Invarsp4},  \ref{tm24} \end{tabular}} & & $\mathcal A = \langle I_4, I_{5a}, \nabla\rangle$ & \begin{tabular}{c} \rule[-0ex]{0pt}{2.1ex}$\{F=0, G=0\}$ \\ $F, G \in \mathcal A$\end{tabular} \\\cline{2-4}
 & \multirow{3}{*}{$\Sigma \colon R_1=0$}  & $\mathcal A_{\Sigma}=\langle  J_8, \nabla_\Sigma \rangle$ & \begin{tabular}{c}\rule[-0ex]{0pt}{2.1ex} $\{F=0, R_1=0\}$ \\ $F \in \mathcal A_{\Sigma}$ \end{tabular}  \\\cline{3-4}
  & & $Q_2$ &\rule[-0ex]{0pt}{2.1ex} $\{Q_2=0, R_1=0\}$   \\\cline{3-4}
  & & $Q_6$ &\rule[-0ex]{0pt}{2.1ex} $\{Q_6=0, R_1=0\}$   \\\cline{2-4}
   & \multirow{2}{*}{$\Pi \colon R_3=0$} & $\mathcal A_{\Pi}=\langle K_5, K_6,  \nabla_\Pi\rangle$ & \begin{tabular}{c}\rule[-0ex]{0pt}{2.1ex} $\{F=0, R_3=0\}$ \\ $F \in \mathcal A_{\Pi}$ \end{tabular}  \\\cline{3-4}
  & & $P_4$&\rule[-0ex]{0pt}{2.1ex} $\{P_4=0, R_3=0\}$   \\\cline{2-4}
    & $y_2=0, z_2=0$ &  &\rule[-0ex]{0pt}{2.1ex} $\{y_2=0, z_2=0\}$  \\\hline
\end{tabular}

\caption{\vphantom{$\dfrac{A}{A}$}%
Algebraic determined ODE systems with finite-dimensional primitive Lie algebras of symmetries. 
Notations: $I_i$ -- absolute differential invariants; $J_i$, $K_i$ -- conditional absolute differential invariants; $R_i$ -- relative differential invariants; $Q_i$, $P_i$ -- conditional relative differential invariants; $\mathcal A$ --  algebra of rational absolute differential invariants with given generators;  $\mathcal A_\Sigma$ -- conditional relative invariants on the underdetermined ODE $\Sigma$.}
\label{tab:primitive}
\end{table}

Our main results about ODE systems are contained in \S \ref{sect:Minkowski}-\ref{sect:Lie29}, and summarized in Table \ref{tab:primitive} and Table \ref{tab:imprimitive}. They are novel and provide new interesting classes of ODE systems. These systems, expressed in terms of invariants, are summarized in theorems at the end of each subsection, following the manner in which Lie often presented his results \cite{L2}. In each subsection we use repeated notation, like $R_i$ for relative invariants, $I_j$ for absolute invariants etc.
These quantities keep the same meaning within the actual subsection but change when we pass to the next one;
this allows us to keep the same strategy of exposition, while avoiding to introduce complicated notations.

\begin{table}[h!]
\begin{tabular}{|c|l|l|l|}
\hline
Lie algebra $\mathfrak g$ & Invariant condition & Scalar invariants & $\mathfrak g$-inv. ODE system $\mathcal E$ \\\hline \hline
\multirow{7}{*}{\begin{tabular}{c} Lie{\hskip1pt}16 \\ \\  Ref: \S\ref{sect:Lie16}\\ Th. \ref{th:Invar16},   \ref{th:invarEq16} \end{tabular}} &  & $\mathcal A = \langle I_2, I_{3a}, \nabla\rangle$ & \begin{tabular}{c} \rule[-0ex]{0pt}{2.1ex} $\{F=0, G=0\}$ \\ $F, G \in \mathcal A$\end{tabular}  \\\cline{2-4}
 & \multirow{2}{*}{$\Sigma \colon R_1=0$}  & $\mathcal A_{\Sigma}=\langle  J_4, \nabla_\Sigma\rangle$ & \begin{tabular}{c} \rule[-0ex]{0pt}{2.1ex} $\{F=0, R_1=0\}$ \\ $F \in \mathcal A_{\Sigma}$ \end{tabular}  \\\cline{3-4}
 && $R_2$ & \rule[-0ex]{0pt}{2.1ex}  $\{R_2=0, R_1=0\}$ \\\cline{2-4}
   & $\Pi \colon R_2=0$ & $\mathcal A_{\Pi}=\langle K_2, \nabla \rangle$ & \begin{tabular}{c} \rule[-0ex]{0pt}{2.1ex} $ \{F=0, R_2=0\}$ \\ $F \in \mathcal A_{\Pi}$ \end{tabular}  \\\hline \hline
\multirow{15}{*}{\begin{tabular}{c} Lie{\hskip1pt}27 \\ \\ \\ \\  Ref: \S\ref{sect:Lie27} \\ Th. \ref{th:Invar27}, \ref{th:invarEq27} \end{tabular}} &  & $\mathcal A = \langle I_3, I_{4a}, I_{4b}, \nabla\rangle$ & \begin{tabular}{c} \rule[-0ex]{0pt}{2.1ex} $ \{F=0, G=0\}$ \\ $F, G \in \mathcal A$\end{tabular}   \\\cline{2-4}
 & \multirow{3}{*}{$\Sigma \colon R_1=0$}  & $\mathcal A_{\Sigma}=\langle J_6, \nabla_\Sigma\rangle$ & \begin{tabular}{c} \rule[-0ex]{0pt}{2.1ex} $ \{F=0, R_1=0\}$ \\ $F \in \mathcal A_{\Sigma}$ \end{tabular}  \\\cline{3-4}
  & & $Q_1$ & $\{Q_1=0, R_1=0\}$  \\\cline{3-4}
  & & $Q_4$ & $\{Q_4=0, R_1=0\}$  \\\cline{2-4}
   & \multirow{3}{*}{$\Pi_a \colon R_{2a}=0$} & $\mathcal A_{\Pi_a}=\langle K_5, \nabla_a\rangle$ & \begin{tabular}{c} \rule[-0ex]{0pt}{2.1ex} $\{F=0, R_{2a}=0\}$ \\ $F \in \mathcal A_{\Pi_a}$ \end{tabular}  \\\cline{3-4}
   & & $R_{2b}$ & \rule[-0ex]{0pt}{2.1ex} $\{ R_{2b}=0, R_{2a}=0\}$ \\\cline{3-4}
   & & $P_4$ &  \rule[-0ex]{0pt}{2.1ex} $\{P_4=0, R_{2a}=0 \}$ \\\cline{2-4}
   & \multirow{3}{*}{$\Pi_b \colon R_{2b}=0$} & $\mathcal A_{\Pi_b}=\langle L_5, \nabla_b\rangle$ & \begin{tabular}{c} \rule[-0ex]{0pt}{2.1ex} $\{F=0, R_{2b}=0\}$ \\ $F \in \mathcal A_{\Pi_b}$ \end{tabular}  \\\cline{3-4}
   & & ($R_{2a}$) & \rule[-0ex]{0pt}{2.1ex} ($\{ R_{2a}=0, R_{2b}=0\})$ \\\cline{3-4}
   & & $T_4$ & \rule[-0ex]{0pt}{2.1ex} $\{T_4=0, R_{2b}=0 \}$  \\\hline  \hline
    \multirow{5}{*}{\begin{tabular}{c} Lie{\hskip1pt}29 \\ \\ Ref: \S\ref{sect:Lie29} \\ Th. \ref{th:Invar29}, \ref{th:invarEq29} \end{tabular}} & & $\mathcal A = \langle I_3, I_{4a}, \nabla\rangle$ & \begin{tabular}{c} \rule[-0ex]{0pt}{2.1ex} $\{F=0, G=0\}$ \\ $F, G \in \mathcal A$\end{tabular} \\\cline{2-4}
 & \multirow{2}{*}{$\Sigma \colon R_2=0$}  & $\mathcal A_{\Sigma}=\langle J_3, \nabla_\Sigma\rangle$ & \begin{tabular}{c}\rule[-0ex]{0pt}{2.1ex}  $\{F=0, R_2=0\}$ \\ $F \in \mathcal A_{\Sigma}$ \end{tabular}  \\\cline{3-4}
  & & $Q_2$ & \rule[-0ex]{0pt}{2.1ex} $\{Q_2=0, R_2=0\}$   \\\hline
\end{tabular}

\caption{\vphantom{$\dfrac{A}{A}$}%
Algebraic determined ODE systems that have finite-dimensional Lie algebras of symmetries that preserve a 1-dimensional foliation.} \label{tab:imprimitive}
\end{table}

Let us also note that the observation we made about the most symmetric ODEs above extends to ODE {\em systems}
(for simplicity restrict to systems of equations of the same order).
Indeed, the maximally symmetric system of order $n$ in $m$ dependent variables $y^\alpha_n=0$
($\alpha=1,\dots,m$; $m>1$) has
fiber-preserving symmetry $\bigl(\mathfrak{sl}(2)\oplus\mathfrak{gl}(m)\bigr)\ltimes\bigl(S^{n-1}\C^2\otimes\C^m\bigr)$
for $n>2$ and has essentially point symmetry algebra $\mathfrak{sl}(m+2)$ for $n=2$;
the symmetry algebra has dimension $d_{\text{max}}=m(m+n+2\delta_{n,2})+3$.
Submaximally symmetric ODE systems with symmetry dimension $d_{\text{submax}}=d_{\text{max}}-2$
were computed in \cite{KT}: they are either linear Wylczynski type
$y^\alpha_n=y^\alpha_{n-r}$ with symmetry $\bigl(\C\oplus\mathfrak{gl}(m)\bigr)\ltimes\bigl(\C^n\otimes\C^m\bigr)$
or in the case of pairs of ODEs ($m=2$) the C-class-type with two different cases:
 $$
\{2y_1y_3=3y_2^2,\ 2y_1z_3=3y_2z_2\} \quad\text{ and }\quad \{y_3=z_2^2,\ z_3=0\}.
 $$
Abstractly the symmetry algebras are, respectively, $(\C\oplus\mathfrak{sl}(2)\oplus\mathfrak{sl}(2))\ltimes(\C^2\otimes\C^2)$
and the graded Lie algebra $\mathfrak{gl}(2)_0\oplus(\C\oplus\C^3)_1\oplus(\C^3)_2$, where we indicate the grading
and the $\mathfrak{sl}(2)$ module type.
All those symmetry algebras with submaximal dimension are fiber preserving.

Our work is based on the realization of Lie algebras by vector fields, and we make several observations related to this at the end of the paper. The choice of setup is important when classifying infinitesimal actions (local, semi-global, etc)
as well as when classifying invariant differential equations (algebraic, analytic, etc). In \S\ref{Sect:Misc} we show, using different realizations of $\mathfrak{sl}(2)$, how the choice of realization and coordinate chart affects the order and other properties of generators for the absolute and relative differential invariants. We also show, on an example
of $\mathfrak{sp}(4)$, that different realizations can be related by a jet-prolongation and a projection, and that this can be interpreted as a twistor correspondence.

We observe that the prolongation of certain Lie algebra actions on an (underdetermined) differential equation of infinite type never achieves freeness,
something that is impossible without a differential constraint (see \cite{AO}). This demonstrates limitations
of the method of moving frames even for finite-dimensional Lie groups.
Finally we will justify the claim that the set of ODEs (scalar or systems) with symmetry algebras that are not essentially fiber-preserving is meager by briefly discussing the moduli space of invariant ODEs with symmetry.

We conclude in \S\ref{OutL} with an overview of the main results and discuss possible generalizations.
In appendix \ref{ApA} we explain the conditions for the symmetry algebra of a scalar ODE or an ODE system
to be finite-dimensional, which justifies our usage of the classification of finite dimensional Lie algebras
of vector fields in the plane and in the space.
Appendix \ref{ApB} is devoted to a brief review of the Sophus Lie classification relevant for our purposes.

\section{Scalar ODEs} \label{sect:Scalar}

We start by summarizing the required concepts and setting the notations, cf.\ \cite{KLV,O}.

Let $J^i=J^i(\mathbb C)$ be the space of jets of (local) functions $\mathbb C(x)\to\mathbb C(y)$,
and let $x,y_0,\dots,y_i$ be the induced coordinates on $J^i$ (we identify $y=y_0$).
Let $\pi_i$ denote the projection $J^i \to \mathbb C$ and let  $\pi_{j,i}$ denote the projection $J^j \to J^i$ for $j> i$.
A scalar ODE of order $k$ can be identified with a submanifold
\[ \mathcal E = \{ F(x,y,\dots,y_k) = 0\} \subset J^k\]
defined by a (local) analytic function\footnote{Here and throughout the paper we use $\mathcal{O}(J^k)$ to denote the ring of (local) analytic functions on $J^k$. Alternatively, one may use it for meromorphic functions on $J^k$ in all our statements.} $F \in \mathcal{O}(J^k)$ with $F_{y_k}\not\equiv 0$.
This implies that in the neighborhood of a generic point in $J^k$ where $F_{y_k}$ is not equal to zero, the equations are normal,
implying that $y_k$ can be expressed locally in terms of a function on $J^{k-1}$.

An (infinitesimal) symmetry of $\mathcal E$ is a vector field on $J^k$ tangent to $\mathcal E$, which preserves
the Cartan distribution on $J^k$. The latter is spanned by $D_x^{(k)}=\p_x+\sum_{i=1}^{k}y_{i}\p_{y_{i-1}}$ and
$\p_{y_{k}}$. By the Lie-B\"acklund theorem all vector fields preserving the Cartan distribution on $J^k$
are prolongations of contact vector fields on $J^1$. Thus, every symmetry of $\mathcal E$ is
a contact vector field $X$ satisfying the condition $X^{(k)}(F)|_{\mathcal E} =0$, where $X^{(k)}$ denotes the prolongation of $X$ to $J^k$. It is a point symmetry if $X$ preserves the fibers of $\pi_{1,0}$ and
fiber-preserving if it also respects the fibers of $\pi_0$ (and hence $\pi_1=\pi_0\circ\pi_{1,0}$);
as in the introduction we use the word ``essential'' to
signify these properties after a possible conjugation by a local contact diffeomorphism.

A Lie algebra of contact vector fields is called irreducible if there exists no invariant foliation by Legendrian curves;
this is equivalent to nonexistence of an invariant subdistribution in the contact distribution. Otherwise it is called reducible. Since a line distribution can be locally rectified, the reducible case corresponds to esentially point and fiber-preserving Lie algebras.

A Lie algebra of vector fields on $J^0$ is called imprimitive if it preserves a 1-dimensional foliation on $J^0$ and
it is called primitive otherwise. Thus the imprimitive case corresponds to essentially fiber-preserving Lie algebras.

Our main question in this section is: Which ODEs have (i) essentially contact, or (ii) essentially point Lie algebras of symmetries? In both cases we refer to the full symmetry algebra. 

\smallskip
\noindent$\bullet$ \textbf{Finite-dimensional irreducible Lie algebras of contact vector fields on $J^1\simeq\mathbb C^3$}
were classified by Lie \cite{L1}, see also \cite[Table 4]{O}.
There are only three such algebras up to local contact transformations, and they all contain the
6-dimensional Lie subalgebra
 \begin{equation}\label{L1}
\langle \p_x, \p_y, x\p_y+\p_{y_1}, x^2\p_y+2x\p_{y_1},
-x\p_x+y_1\p_{y_1}, 2y_1\p_x+y_1^2\p_y \rangle
 \end{equation}
abstractly isomorphic to $\mathfrak{sl}(2)\ltimes\mathfrak{heis}(3)$, a maximal subalgebra in
the parabolic $\mathfrak{p}_1\subset\mathfrak{sp}(4)$.

\smallskip
\noindent$\bullet$ \textbf{Finite-dimensional primitive Lie algebras of point vector fields on $\mathbb C^2$}
were classified by Lie \cite[p.124]{L2}, see also \cite[Table 2]{O}.
There are only three such algebras up to local point transformations, and they all contain the
5-dimensional Lie subalgebra $\mathfrak{saff}(2)=\mathfrak{sl}(2)\ltimes\C^2$ given by
 \begin{equation}\label{L2}
\langle \p_x, \p_y, x\p_y, -x\p_x+y\p_y, y\p_x \rangle .
 \end{equation}

The Lie algebras \eqref{L1} and \eqref{L2} will be treated in \ref{sect:contact} and \ref{sect:point} respectively.

\begin{figure}
\begin{tikzcd}
& &C_2 = \mathfrak{sp}(4) \\
&\mathfrak{sl}(2) \ltimes \mathfrak{heis}(3)\arrow[hook]{rr}\arrow[hook]{ru} & &\mathfrak{p}_1(C_2)=\mathfrak{gl}(2) \ltimes \mathfrak{heis}(3) \arrow[hook]{lu}
\end{tikzcd}
\begin{tikzcd}
& &A_2 = \mathfrak{sl}(3) \\
&\mathfrak{saff}(2)\arrow[hook]{rr}\arrow[hook]{ru} & &\mathfrak{p}_1(A_2)=\mathfrak{aff}(2) \arrow[hook]{lu}
\end{tikzcd}
\caption{Diagram of inclusions of the irreducible contact and primitive point Lie algebras of vector fields in $\mathbb C^2$. }
\label{fig:2Dinclusions}
\end{figure}
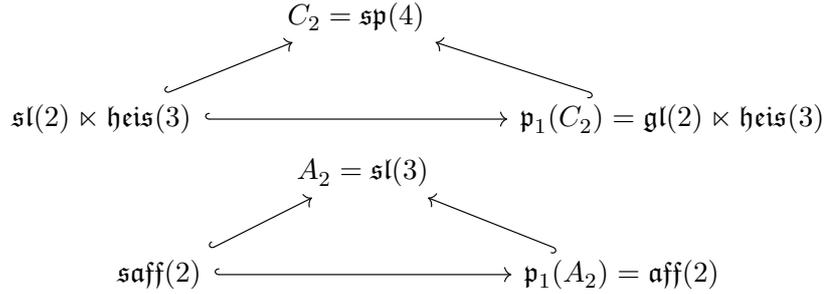

\begin{remark} \label{rk:realization}
In Figure \ref{fig:2Dinclusions} the Lie algebras are labelled by their abstract Lie algebra structure. It is important to keep in mind that they are not only abstract Lie algebras, but realizations as Lie algebras of vector fields on $\mathbb C^3$ and $\mathbb C^2$, respectively. The specific realization is important, as we indicate in \S\ref{sl2}. However, some realizations may become equivalent after prolongation as we demonstrate in \S\ref{sect:twistor}.
\end{remark}

\subsection{Recollection on differential invariants}\label{rode}

Let $\mathfrak g$ be a Lie algebra of point or contact vector fields. The $\mathfrak g$-orbit through $\theta \in J^k$ is the set of all points obtained by the flow of vector fields in $\mathfrak g$.
(This is equivalent to the orbit of the corresponding connected Lie pseudogroup.)
The space $J^k$ is partitioned into $\mathfrak g$-orbits, and any $\mathfrak g$-invariant ODE of order $k$ is a union of $\mathfrak g$-orbits in $J^k$.

 \begin{definition}
A differential invariant of order $k$ is a function on $J^k$ which is constant on $\mathfrak g$-orbits.
 \end{definition}

Under sufficiently general conditions, which are satisfied for all the Lie algebras of vector fields that we consider
in this paper, the $\mathfrak g$-orbits in general position in $J^k$ can be separated
by rational differential invariants of order $k$, see \cite{KL}. These conditions are that the corresponding Lie pseudogroup $G$ is transitive on $J^0$ (or $J^1$ for contact transformations) and that the diffeomorphism subgroups $G_a^k$ are algebraic for every $k \geq i$ and every point $a \in J^i$ ($i=0$ or $i=1$). When these conditions hold, the orbits in $J^k$ can be identified with orbits of the algebraic group $G_a^k$ on the fibers of $J^k\to J^i$ ($i=0$ or $i=1$), and it follows from Rosenlicht's theorem that the latter orbits are separated by rational invariants. It is thus sufficient to consider differential invariants whose restriction to fibers of $J^k \to J^i$ are rational. We will refer to such differential invariants using the adjective ``rational'' and point out that this ``rationality'' property of differential invariants is preserved under point and contact transformations, respectively. All the Lie algebras we consider act transitively, and  for all of them the stabilizer of a point in $J^i$ ($i=0$ or $i=1$) integrates to an algebraic Lie group.

For finite-dimensional Lie algebras, the number of independent differential invariants of order $k$
will grow without bound as $k$ increases,
but the field of rational differential invariants is finitely generated as a differential field:
There exists an invariant derivation acting on the space of invariants.

Differential invariants computed in this paper are  usually found by solving the system
 \[
X^{(k)}(I)=0, \qquad X \in \mathfrak g.
 \]
Similarly, the derivation is found by solving the system
 \[
[X^{(\infty)},f D_x]=0, \qquad X \in \mathfrak g,
 \]
where $f$ is a function on $J^i$ for some $i$ and $D_x=\p_x+\sum_{i=1}^{\infty} y_i \p_{y_{i-1}}$
is the total derivative operator.

 \begin{definition}
A relative differential invariant is a function on $J^k$ with $\mathfrak g$-invariant zero locus.
 \end{definition}

Relative differential invariants can be found by solving the system
 \[
X^{(k)}(R)=\lambda(X)R, \qquad X \in \mathfrak g,
 \]
where $\lambda\in\g^*\otimes\mathcal{F}(J^\infty)$ is called the weight of $R$.
Here by $\mathcal{F}(J^\infty)=\cup_j\mathcal{F}(J^j)$
we mean the appropriate algebra of functions on the space of jets.
We choose rational functions in higher jets $y_k$ ($k>0$) for absolute invariants $I$
(some other choices: analytic/smooth in the complex/real cases, respectively)
and {\em polynomial\/} ones for relative invariants $R$, so that
$\lambda\in\g^*\otimes\mathcal{P}(J^\infty)$.

To facilitate computations we used Maple. Often, the results of pdsolve are not rational functions.
They may not even be invariant (that is, not constant on orbits, see Remark \ref{rmk:derivation}).
However, they  can still be used to generate a transcendence basis for the field of rational invariants.
Having a transcendence basis is in general not sufficient for generating the whole field of differential invariants.
But in the cases we consider one can show
that the field generated by the transcendence basis has
no algebraic field extensions within the field of rational invariants.

An approach due to Sophus Lie to find fundamental relative invariants of Lie group actions is as follows.
 An effective action of a finite-dimensional Lie group is eventually free (see \cite{Ovs} or \cite{AO}). In fact, for an (effective) Lie group $G$
freeness of the action is attained on $J^s$ where $s=\dim G-2$, except for particular actions of $\text{Sol}(2) \ltimes \mathbb C^{k}$ where a pseudo-stabilization effect happens (see page 202 and 1.7b of Table 5 \cite{O}) in which case $s=\dim G-1$.  The Lie determinant is the evaluation of a volume form
on $J^s$ on the prolonged vector fields $X^{(s)}_t$, where $\{X_t\}_{t=1}^{\dim G}$ is a basis of the Lie algebra of vector fields corresponding to the $G$-action; in the case of pseudo-stabilization this is not a scalar, but a 1-form proportional to the differential of an absolute invariant.
The irreducible factors of this {\em algebraic\/} (in higher jets) expression are relative invariants $R_j$.

The following theorem is a slightly improved version of \cite[Theorem 6.36]{O}.

 \begin{theorem}[Sophus Lie]
For an algebraic Lie group action, an invariant ODE $\E=\{F=0\}\subset J^k$ is either given by fundamental relative
invariants $R_j=0$ or can be expressed through absolute differential invariants $F=f(I_1,\dots,I_r)$,  where $I_1,\dots, I_r$ are generators for the field of rational absolute differential invariants of order $k$.
 \end{theorem}

The function $f$ is rational or analytic (meromorphic), depending on the setup.

\subsection{ODEs with essentially contact symmetries} \label{sect:contact}

From the discussion above, we obtain the following description for ODEs of order greater than 2.
The restriction on the order is made in order to avoid ODEs having infinite-dimensional Lie algebras of contact symmetries (see Appendix \ref{ApA}).

\begin{prop}
Assume that the Lie algebra of contact symmetries of a scalar ODE of order\linebreak
 $k>2$ is irreducible.
Then it contains a Lie subalgebra $\g$ that, up to a local contact transformation, is given by \eqref{L1}.
 \end{prop}

Absolute differential invariants of the action of $\g$ on $J^\infty$ are generated (in Lie-Tresse sense \cite{KL})
by one differential invariant and one invariant derivation (see also \cite[Table~5 \#4.1]{O}).
They can be given in terms of the relative differential invariants, the simplest of which are
 \begin{align*}
R_3&= y_3, &\qquad
R_5 &= 3 y_3 y_5-5 y_4^2,\\
R_6 &= 9 y_3^2 y_6 -45 y_3 y_4 y_5+40 y_4^3, &\qquad
R_7 &= 9y_3^3 y_7 -63 y_3^2 y_4 y_6+105 y_3 y_4^2 y_5-35 y_4^4,\\
R_8 &= 9y_3^4y_8-84y_3^3y_4y_7 + &&\hspace{-87pt} 210y_3^2y_4^2y_6-105y_3^2y_4y_5^2+210y_3y_4^3y_5-280y_4^5.
 \end{align*}

The space of weights is one-dimensional and is generated by the cocycle
$\lambda(X)=-L_X(\alpha)/\alpha$, where $\alpha=dx$.
Note that $\lambda(X_5)=1$ for the 5th basis element $X_5^{(\infty)}=-x\p_x+\sum ny_n\p_{y_n}$ of $\g$,
and so a relative invariant has weight $w\lambda$ (or simply $w\in\mathbb{N}$) if the number of differentiations by $x$ in
each its monomial is $w$. Denote the space of relative invariants of weight $w$ by $\mathcal{R}^w$.
Weights of the above invariants $R_n$ are $w=3,8,12,16,20$, respectively (subscript $n$ denotes the order).

 \begin{theorem}
The field of rational absolute differential invariants is generated by the differential invariant $I_5= \frac{R_5^3}{R_3^8}$ and the invariant derivation $\nabla=\frac{R_5}{R_3^3} D_x$.
 \end{theorem}

With these data one can generate additional algebraically independent invariants, one in each order $>5$.
For example, we have the following invariant of order 6:
 \[
\frac{\nabla(I_5)}{I_5} = \frac{9 y_3^2 y_6-45 y_3 y_4 y_5 +40 y_4^3}{y_3^4}.
 \]
The rational absolute differential invariants separate $\g$-orbits in general position in $J^k$.
In this sense, generic scalar ODE of order $k$ is given by a function of absolute invariants:
 \begin{equation}
f(I_5, \nabla(I_5),\dots,\nabla^{k-5}(I_5))=0. \label{eq:genericcontact}
 \end{equation}
The function $f$ in \eqref{eq:genericcontact} is either rational or analytic, depending on the setup.

The $\mathfrak g$-action is transitive on $J^2$. On $J^3$ there is one 4-dimensional orbit $\Sing^3=\{R_3=0\}$,
and its complement is a 5-dimensional orbit. For $k\geq 4$ the orbits in $J^k$ are 6-dimensional outside the set
$\Sing^k=\pi_{k,3}^{-1}(\Sing^3)$.
The only invariant ODE lying inside $\Sing^k$ is the one given by $R_3=0$
together with its total derivatives. Outside $\Sing^k$, invariant ODEs are given by \eqref{eq:genericcontact}.

 \begin{cor} \label{cor:contact}
Up to a contact transformation every scalar ODE of order greater than 2 with essentially contact symmetry algebra is either
$R_3=0$ or is given by formula \eqref{eq:genericcontact}.
 \end{cor}

In particular, there is only one ODE of order 3 with essentially contact symmetry algebra, namely $y_3=0$,
while there are no such ODEs of order 4. A connected component of a fifth-order ODE $f(I_5)=0$
is given by $I_5=c$ for some $c\in\C$, with  Zariski closure $R_5^3=c R_3^8$, and the constant $c$ can be
normalized to either $0$ or $1$ by rescaling $y$. Therefore, up to contact equivalence,
there are only three (connected) ODEs of order $3, 4, 5$ with essentially contact symmetry algebra:
 \begin{equation}\label{ODe1}
y_3=0, \qquad 3 y_3 y_5-5 y_4^2=0, \qquad (3 y_3 y_5-5 y_4^2)^3=y_3^8.
 \end{equation}
Their full symmetry algebras are $\mathfrak{sp}(4)$,
$\mathfrak{gl}(2)\ltimes\mathfrak{heis}(3)=\mathfrak{p}_1\subset\mathfrak{sp}(4)$ and
$\mathfrak{sl}(2)\ltimes\mathfrak{heis}(3)\subset\mathfrak{p}_1$, respectively.

 \begin{remark}
This discussion also shows that the description of ODEs with essentially contact symmetry algebra in Corollary \ref{cor:contact} is not minimal: it contains several contact-equivalent ODEs.
 \end{remark}


To get algebraic invariant equations in simpler terms, we now describe relative differential
invariants of Lie algebra \eqref{L1}.

 \begin{theorem}\label{TT1}
The graded algebra $\mathcal{R}=\oplus_{w>0}\mathcal{R}^w$ of relative differential invariants
wrt $\mathfrak g$ is generated by $R_3, R_5$ and the relative invariant derivation
$\nabla_w=y_3 D_x-\tfrac{w}{3} y_4:\mathcal{R}^w\to\mathcal{R}^{w+4}$ in the following sense:
any $R\in\mathcal{R}$ is a polynomial combination of $R_3$, $R_5$ and their invariant derivations,
possibly divided by a power of $R_3$.
 \end{theorem}

The term ``graded algebra'' above means that $\mathcal{R}^w$ are vector spaces and
$\mathcal{R}^{w_1} \cdot\mathcal{R}^{w_2} \subset\mathcal{R}^{w_1+w_2}$,
but only homogeneous elements of $\mathcal{R}$ are relative invariants.
The algebra $\mathcal{R}$ is filtered by the jet-order, and for the filtrand $\mathcal{R}_n$
of order $\leq n$ invariants we have: $\mathcal{R}_n=\oplus_{w>0}\mathcal{R}_n^w$.

 \begin{proof}
The Lie determinant of $\mathfrak g$, obtained from the $6\times6$ matrix having
the coefficients of the 4th prolongation of a basis in $\g$ as entries,
is equal to $36y_3^3$. Starting from order 4 the action is
free in a Zariski-open set and there is a bijection between absolute and
relative invariants, one generator in each order.

Any absolute differential invariant is a rational function of the basic one and its invariant derivatives,
but due to stabilization of singularities (see \cite{KL}) the generators for absolute differential invariants can be chosen to have a power of $R_3$ as denominator.
 \end{proof}

The relative invariant derivation can be expressed through the absolute invariant derivation
 \[
\nabla_w(R)=\frac{R_3^{4+w/3}}{R_5}\,\nabla\left(\frac{R}{R_3^{w/3}}\right)
 \]
and the above invariants are generated from $R_3,R_5$:
 $$
\nabla_3 R_3=0,\quad \nabla_8 R_5=\tfrac13R_6,\quad
\nabla_{12} R_6=R_7-5R_5^2,\quad \nabla_{16} R_7=R_8-\tfrac73R_5R_6.
 $$

 \begin{remark}
We can relate these relative invariants to the differential operators from Introduction:
(i) $R_5=L_4(y_1)$;
(ii) $R_6=L_5(y_1)$;
(iii) $L_7(y)=\frac19(10R_7-49R_5^2)$.
 \end{remark}

If we require the equation $\E$ to be algebraic, then $f$ in \eqref{eq:genericcontact} is rational, and
the Zariski closure of the ODE is given by a polynomial equation. Hence we conclude:

 \begin{cor}
The irreducible algebraic ODEs with essential contact symmetry are given by the  homogeneous elements of $\mathcal{R}$.
 \end{cor}

We can effectively describe elements of $\mathcal{R}$.
For instance, let us derive all invariant ODEs of order $n=5$. Relative invariants of weight $w=3s+8t$
have the form $\sum_{-t/3< i\leq s/8}c_i R_3^{s-8i}R_5^{3i+t}$. This
factorizes into a product of terms of type $R_3$, $R_5$ and
$R_5^3-cR_3^8$ (and $c\neq0$ normalizes to $c=1$).
This yields the already observed invariant ODEs \eqref{ODe1}.

For higher order ODEs the situation is more complicated.
The above normalization is due to the PID property for polynomials in one
variable. This property fails for polynomials with more variables, and
for any $n>5$ a generic polynomial in $\mathcal{R}^w_n$ is irreducible.
In particular, there are infinitely many non-equivalent ODEs of orders
$n>5$ with essentially contact symmetry.

For instance, let us construct relative invariants using homogeneous combinations of $R_3$, $R_5$  and $R_6$:
 \[
\sum_{3r+8s+12t=w} C_{rst} R_3^{r} R_5^s R_6^t.
 \]
These combinations can factorize with a power of $R_3$ as one of the factors.
For example, we have
 \[
64 R_5^3+45 R_6^2= 9R_3 R_6',
 \]
where the second factor is a relative invariant of weight $w=21$
 \[
R_6'=45 y_3^3 y_6^2-450 y_3^2 y_4 y_5 y_6+192 y_3^2 y_5^3+400 y_3 y_4^3 y_6+165 y_3 y_4^2 y_5^2-400 y_4^4 y_5
 \]
and it is the numerator of the absolute invariant
 \[
64 I_5+5 \left(\frac{\nabla(I_5)}{I_5} \right)^2.
 \]
The invariant $R_6'$ is not generated algebraically by $R_3, R_5$ and $R_6$, but it appears via localization
(division by $R_3$). This is precisely what we observed in Theorem \ref{TT1}.


\subsection{ODEs with essentially point symmetries} \label{sect:point}

In this section we consider ODEs whose Lie algebra of symmetries is essentially point (and not essentially fiber-preserving). From the discussion in \S\ref{rode} we obtain the following description for ODEs of order at least 2.
Again, the restriction on the order is to assure finite-dimensionality of the symmetry algebra.

 \begin{prop}
Assume that the symmetry algebra of a scalar ODE of order $k\ge2$ is point and primitive.
Then it contains a Lie subalgebra $\g$ that, up to a local point transformation, is given by formula \eqref{L2}.
 \end{prop}

Generators for the field of rational absolute differential invariants with respect to $\mathfrak g$ can be given in terms of the relative differential invariants (see also \cite[Table~5 \#2.1]{O}). The simplest of those are:
 \begin{align*}
R_2 &= y_2,  &\qquad R_4 &= 3y_2y_4-5y_3^2, \\
R_5 &= 9y_2^2y_5-45y_2y_3y_4+40y_3^3,  &\qquad
R_6 &= 9y_2^3y_6-63y_2^2y_3y_5+105y_2y_3^2y_4-35y_3^4,\\
R_7 &= 9y_2^4y_7-84y_2^3y_3y_6+210y_2^2y_3^2y_5- &&\hspace{-50pt}
105y_2^2y_3y_4^2+210y_2y_3^3y_4-280y_3^5.
 \end{align*}

The space of weights is again one-dimensional but now it is generated
by the cocycle $\lambda(X)=\frac12L_X(\beta)/\beta$, where $\beta=dx\wedge dy$.
Note that $\lambda(X_4)=1$ for the 4th basis element $X_4=-x\p_x+\sum (n+1)y_n\p_{y_n}$ of $\g$,
and so a relative invariant has weight $w\lambda$ (or simply $w\in\mathbb{N}$) if
the number of differentiations by $x$ plus the number of $y$ in each its monomial is $w$. The space of
relative invariants with this weight $w$ will be denoted by $\mathcal{R}^w$.
The weights of the above relative invariants $R_n$ are $3,8,12,16,20$, respectively (index $n$ is the order).

 \begin{theorem}\label{TT2}
The field of rational absolute differential invariants is generated by the differential invariant
$I_4 = \frac{R_4^3}{R_2^8}$ and the invariant derivation $\nabla=\frac{R_4}{R_2^3} D_x$.
 \end{theorem}

The $\mathfrak g$-action is transitive on $J^1$. On $J^2$ there is one 3-dimensional orbit $\Sing^2=\{R_2=0\}$,
and its complement is a 4-dimensional orbit. For any $k\ge3$ the orbits in $J^k$ are 5-dimensional outside the
set $\Sing^k=\pi_{k,2}^{-1}(\Sing^2)$. The only invariant ODE lying inside $\Sing^k$ is the one given by the equation $R_2=0$ and its total derivatives. Outside $\Sing^k$, invariant ODEs are given by
 \begin{equation}\label{eq:genericpoint}
f(I_4,\nabla(I_4),\dots, \nabla^{k-4}(I_4))=0.
 \end{equation}
Again, the function $f$ is either rational or analytic, depending on the setup.

 \begin{cor}
Up to a point transformation every ODE of order greater than 1 with essentially point symmetry algebra is either $R_2=0$ or is given by
\eqref{eq:genericpoint}.
 \end{cor}

Similar to \S\ref{sect:contact} we conclude that the only 2nd order ODE with essentially point symmetry algebra
is trivializable, namely $y_2=0$, while there are no such ODE of order 3. Furthermore, all connected ODEs of order  $2,3,4$ with essentially point symmetry algebra are equivalent to one of the following:
 \begin{equation}\label{ODe2}
y_2=0, \qquad 3 y_2 y_4 - 5 y_3^2 =0, \qquad (3 y_2 y_4-5 y_3^2)^3=y_2^8.
 \end{equation}
Their full symmetry algebras are $\mathfrak{sl}(3)$, $\mathfrak{aff}(2)$ and
$\mathfrak{saff}(2)$, respectively.



To get the algebraic invariant equations in simpler terms, we now describe relative differential
invariants of Lie algebra \eqref{L2}.

 \begin{theorem}
The algebra $\mathcal{R}=\oplus_{k>0}\mathcal{R}^k$ of relative differential invariants
wrt $\g$ is generated by $R_2$, $R_4$ and the relative invariant derivation
$\nabla_w=y_2 D_x-\tfrac{w}{3} y_3:\mathcal{R}^w\to\mathcal{R}^{w+4}$ in the following sense:
any $R\in\mathcal{R}$ is a polynomial combination of $R_2$, $R_4$ and their invariant derivations,
possibly divided by a power of $R_2$.
 \end{theorem}

The proof is similar to that of Theorem~\ref{TT1} and is therefore omitted.

The relative invariant derivation can be related to the absolute invariant derivation
 \[
\nabla_w(R)=\frac{R_2^{4+w/3}}{R_4}\,\nabla\left(\frac{R}{R_2^{w/3}}\right)
 \]
and the above relative invariants are generated from $R_2,R_4$ in the following way:
 $$
\nabla_3 R_2=0,\quad \nabla_8 R_4=\tfrac13R_5,\quad
\nabla_{12} R_5=R_6-5R_4^2,\quad \nabla_{16} R_6=R_7-\tfrac73R_4R_5.
 $$

 \begin{cor}
Algebraic ODEs with essential point symmetry are given by  homogeneous elements of $\mathcal{R}$.
 \end{cor}

Such ODEs can be effectively described as above.
Again, up to order 4 there are finitely many non-equivalent ODEs with essentially point symmetry algebra:
actually only $R_2=0$, $R_4=0$ and $R_4^3-R_2^8=0$, corresponding to equations \eqref{ODe2}.
Starting from order 5 there are infinitely many non-equivalent ODEs
with essentially point symmetry.


 \begin{remark}\rm
The similarity between computations in this and the previous section
has a conceptual explanation.
The contact vector fields in \eqref{L1} preserve the distribution $\langle\p_y\rangle$ transversal to the contact distribution.
Hence the projection along it is a homomorphism from Lie algebra \eqref{L1} to \eqref{L2} with one-dimensional kernel:
wipe out every occurrence of $\p_y$ in the first Lie algebra and then change $y_1$ to $y$. This
decreases the order by 1 and is compatible with prolongations, see \cite[Sect.\ 6.2-6.3]{DK}.
Thus the algebra $\mathcal{R}$ of Theorem \ref{TT1} will pass to that of Theorem \ref{TT2}.
 \end{remark}

\section{Systems of ODEs whose symmetries are not fiber preserving}\label{sect:Sytems}

In this section we turn to systems of ordinary differential equations on two functions of one variable.
We use the notation $J^i=J^i(\C,\C^2)$ for the space of jets of a pair of functions on $\C$.
Coordinates $(x,y,z)$ on $J^0=\C\times\C^2$ induce coordinates $\bigl(x,\{y_i\}_{i=0}^k,\{z_j\}_{j=0}^k\bigr)$ on $J^k$.

By a (determined) {\it system of ODEs} of orders $(k,l)$, $k\ge l$, we mean a system given by two functions
$F\in \mathcal{O}(J^k), G\in \mathcal{O}(J^l)$
with the property
 \[
\left|\begin{matrix} F_{y_k} & F_{z_k} \\ G_{y_l} & G_{z_l} \end{matrix}\right| \neq 0.
 \]
Such a pair defines a $(k+l+1)$-dimensional submanifold
 \begin{equation}
\E = \{F=0,G=0,D_x(G)=0,\ldots,D_x^{k-l} (G)=0\} \subset J^k. \label{eq:GeneralSystem}
 \end{equation}


A symmetry of $\E$ is a vector field on $J^k$ tangent to $\E$, which preserves the Cartan distribution.
The latter is spanned by $\p_{y_{k}}$, $\p_{z_{k}}$ and $D_x^{(k)}=\p_x+\sum_{i=1}^{k} \left( y_{i}\p_{z_{i-1}}+y_{i}\p_{z_{i-1}} \right)$. By the Lie-B\"acklund theorem all vector fields preserving the Cartan distribution on $J^k$
are prolongations of point fields on $J^0(\C,\C^2)$, though this is not true for mixed jets\footnote{In the space of jets
of mixed order $J^{k,l}$, $k>l$, with coordinates $\bigl(x,\{y_i\}_{i=0}^k,\{z_j\}_{j=0}^l\bigr)$ an analog of the
Lie-B\"acklund theorem \cite{AK} allows prolongations of parameter-dependent contact vector fields,
but for the sake of simplicity we do not consider those in this paper.}.

A vector field $X$ on $J^0$ is a point symmetry of $\E$ if $X^{(k)}(F)|_{\E}=0$ and $X^{(l)}(G)|_{\E}=0$.
The latter condition implies $X^{(l+i)}(D_x^iG)|_{\E}=0$ for $i=1,\dots,k-l$. The point symmetries make up a Lie algebra.

If $k, l \geq 2$ this Lie algebra is always finite-dimensional, while if $l=1$ additional conditions must be satisfied for it to be finite-dimensional (see Appendix \ref{ApA} for details). Since we are relying on the classification of finite-dimensional Lie algebras of vector fields, our description of invariant ODE systems is complete only under these conditions.

We split the finite-dimensional Lie algebras of point vector fields into two classes: those that preserve a 2-dimensional foliation in $\mathbb C^3$ and those that don't. The significance of the first class is that its members are conjugate by point transformations to Lie algebras of vector fields that preserve the fibers of $J^0$. The second class can be further split into two classes:

\smallskip
\noindent$\bullet$ \textbf{The primitive Lie algebras of vector fields on $\C^3$}
preserve no foliation and were classified by Lie in \cite[Chapt. 7]{L2}.
There are only 8 primitive Lie algebras, and they all have one of the following three as a Lie subalgebra:
  \begin{align}
\circ\quad & \langle \partial_x, \partial_y, \partial_z, x \partial_y+y \partial_x, x \partial_z+z \partial_x, y \partial_z - z \partial_y \rangle,\label{L3}\\
\circ\quad  & \langle \partial_y, \partial_z, 2 \partial_x+y \partial_y+z\partial_z, y\partial_y-z\partial_z, 2 y \partial_x+y^2 \partial_y-e^x \partial_z,2 z \partial_x-e^x \partial_y+z^2 \partial_z\rangle,\label{L4}\\
\circ\quad & \langle \partial_x, \partial_y-z \partial_x, \partial_z+y \partial_x, y\partial_z, z\partial_y, y \partial_y-z\partial_z, 2x \partial_x+y \partial_y+z \partial_z, \notag\\
& \qquad  xy \partial_x+y^2\partial_y+(yz+x) \partial_z, x z \partial_x+ (yz-x)\partial_y+z^2\partial_z,x ( x\partial_x+y \partial_y+z\partial_z) \rangle.\label{L5}
  \end{align}
Figure \ref{fig:Primitive} shows a diagram of inclusions between the eight primitive Lie algebras.

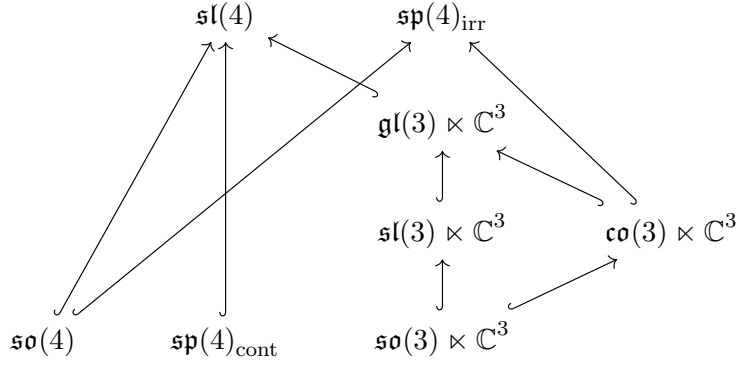
\begin{figure}
\begin{tikzcd}
& &\mathfrak{sl}(4) & \mathfrak{sp}(4)_{\text{irr}} \\
& & &\mathfrak{gl}(3)\ltimes \mathbb C^3\arrow[hook]{ul} \\
& & &\mathfrak{sl}(3)\ltimes \mathbb C^3 \arrow[hook]{u} & \mathfrak{co}(3)\ltimes \mathbb C^3 \arrow[hook]{uul}\arrow[hook]{ul} \\
&\mathfrak{so}(4)\arrow[hook]{uuurr} \arrow[hook]{ruuu} &\mathfrak{sp}(4)_{\text{cont}} \arrow[hook]{uuu} & \mathfrak{so}(3)\ltimes \mathbb C^3 \arrow[hook]{u}\arrow[hook]{ur}
\end{tikzcd}
\caption{Diagram of inclusions of the primitive Lie algebras of vector fields on $\mathbb C^3$. The maximal elements are collected on the top, while the minimal elements are collected on the bottom. Of these eight Lie algebras, only 2 share the same Lie algebra structure. To distinguish them, we denote them here by $\mathfrak{sp}(4)_{\text{irr}}$ (preserves no distribution) and $\mathfrak{sp}(4)_{\text{cont}}$ (preserves the contact distribution).}
\label{fig:Primitive}
\end{figure}

The Lie algebra \eqref{L3} preserves the Minkowski metric $-dx^2+dy^2+dz^2$ while \eqref{L4} preserves
the metric $dx^2+4 e^{-x} dy dz$. Both of these metrics have constant curvature and the algebras are
isomorphic to $\mathfrak{so}(3)\ltimes\C^3$ and $\mathfrak{so}(4)$, respectively.
The Lie algebra \eqref{L5} is the projectivization of the linear $\mathfrak{sp}(4)$-action on $\C^4(t,x,y,z)$
with the symplectic form $dt \wedge dx + dy \wedge dz$. It can also be described as the
the Lie algebra of symmetries of the scalar ODE $u'''=0$. However, this is not directly obvious from the coordinate expression since the contact distribution preserved by \eqref{L5} is given by the 1-form $dx-z dy+y dz$ (i.e. not the standard coordinates on $J^1$). In \S\ref{sect:twistor} we explain how this action is related to a different $\mathfrak{sp}(4)$-action on $\mathbb C^3$ through a twistor correspondence.

\smallskip
\noindent$\bullet$ \textbf{The Lie algebras preserving a 1-dimensional foliation, but no 2-dimensional foliation}
are contained among the Lie algebras listed by Lie in \cite[Chapt. 8, §41-§44]{L2}. They all contain a Lie subalgebra which is locally equivalent to one of the following three Lie algebras of vector fields:
  \begin{align}
\circ\quad & \langle \partial_x, \partial_y+x \partial_z, x \partial_y + \tfrac{1}{2} x^2 \partial_z, x \partial_x-y \partial_y, y \partial_x+\tfrac{1}{2} y^2 \partial_z, \partial_z \rangle,\label{L6} \\
\circ\quad &\langle \partial_x,\partial_y, x \partial_y+\partial_z, x \partial_x -y\partial_y-2 z \partial_z,y\partial_x-z^2 \partial_z, x \partial_x+y\partial_y, \notag\\
& \qquad x^2 \partial_x+xy\partial_y+(y-xz)\partial_z, xy\partial_x+y^2 \partial_y+z(y-xz) \partial_z \rangle, \label{L7}\\
\circ\quad & \langle \partial_x ,\partial_y ,x \partial_y ,x \partial_x -y \partial_y ,y \partial_x ,x \partial_x +y \partial_y +\partial_z,\notag\\
& \qquad
x^2 \partial_x +x y \partial_y +\tfrac{3}{2} x \partial_z,x y \partial_x +y^2 \partial_y +\tfrac{3}{2} y \partial_z \rangle.\label{L8}
  \end{align}
  In Appendix \ref{ApB} we will explain why only these three Lie algebras, which we  refer to as Lie{\hskip1pt}16, Lie{\hskip1pt}27 and Lie{\hskip1pt}29, remained from Lie's list of 21 items.

The Lie algebra \eqref{L6} is 6-dimensional and projects to the 5-dimensional Lie algebra of area-preserving vector fields on $\C^2$;
abstractly it is isomorphic to $\mathfrak{sl}(2)\ltimes\mathfrak{heis}(3)$.
The Lie algebras \eqref{L7} and \eqref{L8} are 8-dimensional and project to the $\mathfrak{sl}(3)$-action on $\C^2$.
Considering all three cases as fiber-preserving transformations on the bundle $\C^2(x,y)\times\C(z)\to\C^2(x,y)$,
we observe that the second Lie algebra preserves a projective structure on the fibers,
while the first and last of those preserve a metric structure on the fibers, see \cite{S}.

Our goal is to describe the systems of ODEs that have finite-dimensional symmetry algebras that do not preserve
a 2-dimensional foliation. These ODEs are all invariant under one of the six Lie algebras listed above (in suitable local coordinates). We analyse these cases, one by one, in sections \ref{sect:Minkowski} to \ref{sect:Lie29}. The results are summarized in Tables \ref{tab:primitive} and \ref{tab:imprimitive} in the introduction.

\begin{figure}[h]
\begin{tikzcd}
& & \mathfrak{gl}(3) \ltimes \bigoplus_{k=0}^h S^k \mathbb C_2 \\
&\mathfrak{sl}(3)_1 \ltimes \bigoplus_{k=0}^h S^k \mathbb C_2 \arrow[hook]{ur}  & & ((\mathfrak{gl}(2) \ltimes \mathbb C^2) \oplus \mathbb C)\ltimes \bigoplus_{k=0}^h S^k \mathbb C_2 \arrow[hook]{ul}    \\
& &(\mathfrak{gl}(2) \ltimes \mathbb C^2) \underset{a}{\ltimes}\bigoplus_{k=0}^h S^k \mathbb C_2 \arrow[hook]{ur}\arrow[hook]{ul}[swap]{a=\frac{2h}{3}} &((\mathfrak{sl}(2) \ltimes \mathbb C^2) \oplus \mathbb C)\ltimes \bigoplus_{k=0}^h S^k \mathbb C_2 \arrow[hook]{u}        \\
&\mathfrak{gl}(2) \ltimes \mathfrak{heis}(3) \arrow[hook]{uu}[swap]{h=3}\arrow[hook]{ur}{a=2}  & (\mathfrak{sl}(2) \ltimes \mathbb C^2) \ltimes \bigoplus_{k=0}^h S^k \mathbb C_2\arrow[hook]{u} \arrow[hook]{ur} &\mathfrak{gl}(3)    \\
&\mathfrak{sl}(2) \ltimes \mathfrak{heis}(3)\arrow[hook]{u}\arrow[hook]{ur} &\mathfrak{sl}(3)_2 & \mathfrak{sl}(3)_1 \arrow[hook]{u}
\end{tikzcd}
\caption{Diagram of inclusions for the Lie algebras listed in Proposition \ref{prop:imprimitive}. All inclusions hold for any fixed $h \geq 1$, except for the one that is marked. Different subscripts of $\mathfrak{sl}(3)$ refer to different realizations of the Lie algebra, subscript $a$ under semidirect product encodes the action of $\mathfrak{z}(\mathfrak{gl}(2))$, while $\C_2$ means $(\C^2)^*$.}
\label{fig:imprimitiveBig}
\end{figure}
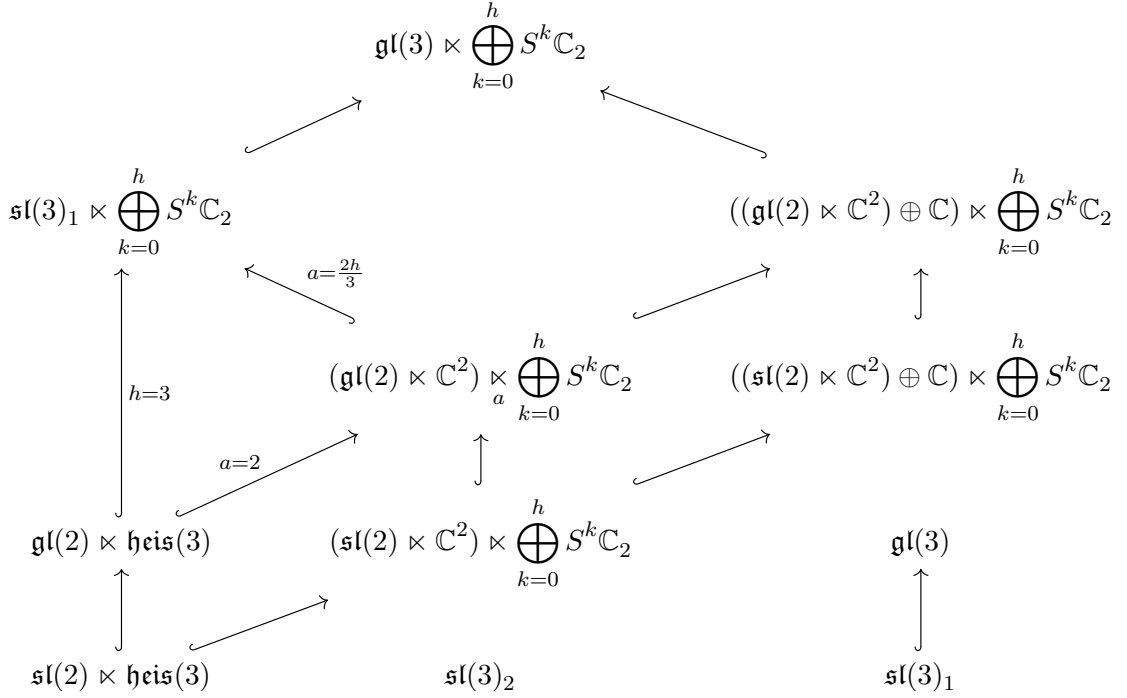

\subsection{A strategy for listing all invariant systems} \label{sect:algorithm}

The analysis of \S\ref{sect:Scalar} for scalar ODEs was relatively simple due to the fact that any scalar ODE
of order $k$ is a submanifold in $J^k$ of codimension 1. Therefore it is given by a function of rational absolute
differential invariants or by a scalar (polynomial) relative differential invariant.
While generic systems of ODEs can be given by two functions of absolute differential invariants, finding the remaining (singular) invariant determined systems requires significantly more effort.

Let $\mathfrak g$ be a Lie algebra of point symmetries of an ODE system of the form \eqref{eq:GeneralSystem}.
If $F$ and $G$ are of the same order $k=l$, then it is possible that the system cannot be defined
as the locus of two scalar relative or absolute differential invariants.
Indeed, there exist linear systems\footnote{For instance, such is the Liouville invariant of
scalar ODEs with cubic dependence on the first derivative (projective connections)
that is responsible for their trivialization by a point transformation, see \cite{Kr}.}
$\{F=0,G=0\}\subset J^k$ of codimension 2 for which
 \[
X^{(k)}(F)= \lambda_1(X) F+\lambda_2(X) G, \quad X^{(k)}(G)= \mu_1(X) F+\mu_2(X) G,
\qquad X\in\g,
 \]
and $\begin{pmatrix}\lambda_1 & \lambda_2\\ \mu_1 & \mu_2\end{pmatrix}\in\g^*\otimes
\mathfrak{gl}(2,\mathcal{P}(J^\infty))$ is a genuine matrix-valued cocycle.

If $F$ and $G$ have different orders $k>l$, then $G$ is either a function of absolute invariants, or a relative invariant.
This is because the vanishing of $X^{(k)}(G)= X^{(l)}(G)$ on $\E$ is independent of the vanishing of $F$.
In that case, $F$ is a conditional relative invariant or a function of conditional absolute differential invariants, defined as follows (see related definitions in \cite[p.~453]{B} and \cite[Def.~4]{Y}).

 \begin{definition}
A function $f$ on $J^k$ is a conditional absolute/relative differential invariant of order $k$ (with respect to the invariant differential equation given by the function $G$ of order $l \leq k$)  if $f$ is an absolute/relative invariant of the $\mathfrak g$-action on $\{G=0, ..., D_x^{k-l} G=0\} \subset J^k$.
 \end{definition}

In some cases the conditional invariants can be interpreted as restrictions of invariants on $J^k$
while in other cases they can not. The following two Lie algebra actions on the plane $\mathbb C^2(x,y)$ illustrate
this phenomenon:
 \begin{itemize}
\item The Lie algebra $\langle x\p_x\rangle$ has invariant $I=y$. The singular set where the orbit dimension drops
is given by $\Sigma=\{x=0\}$, and (the restriction of) $I$ separates orbits on $\Sigma$. The (Zariski) closure of each
1-dimensional orbit contains a 0-dimensional orbit.
\item The Lie algebra $\langle y\p_x \rangle$ has the same absolute invariant $I$.
However, on the singular set $\Sigma=\{y=0\}$, where the orbit dimension drops, the orbits are separated by
the conditional invariant $J=x$. All 1-dimensional orbits are closed.
 \end{itemize}

Our procedure for finding invariant systems of ODEs consists of the following steps:
\begin{enumerate}
  \item {\it Find generators $I_1,\dots, I_r$ and $\nabla$ for the field of rational differential invariants.} The field separates orbits in $J^i \setminus \Sing^i$, where $\Sing^i \subset J^i$ is a $\mathfrak g$-invariant  Zariski-closed subset which by \cite{KL} stabilizes at some order $m$: for $i>m$ we have $\Sing^i= \pi_{i,m}^{-1}(\Sing^m)$. For generic  invariant systems of type \eqref{eq:GeneralSystem}, the functions $F$ and $G$ can be expressed through rational absolute invariants.

  Let $\mathcal E \subset J^k$ be a system of type \eqref{eq:GeneralSystem}. If $\mathcal E \cap \Sing^k \neq \emptyset$, then we can split the system into two components: $\mathcal E \cap \Sing^k$ and $\mathcal E \setminus (\mathcal E \cap \Sing^k)$. Both components are invariant. The first component is not in general a determined system, but it may contain invariant determined systems of lower order.
  \item {\it Determine the singular sets $\Sing^i$.} 
 These sets contain the orbits in $J^i$ where the orbit dimension drops, but they may be larger.  This is due  to the fact that the closures of two (or more) orbits may intersect in the same orbit of lower dimension. In such a situation, the rational absolute differential invariants clearly take the same value on these orbits, and these orbits can therefore not be separated by the absolute invariants.

Let $X_1,\dots, X_q$ be a basis of vector fields in the Lie algebra $\mathfrak g$. Let   $r$ be the dimension of a generic $\mathfrak g$-orbit in $J^i$, meaning that the dimension of the subspace $\langle X_1^{(i)}|_\theta ,\dots , X_q^{(i)}|_\theta  \rangle \subset T_\theta J^i$ is $r$ for a point $\theta \in J^i$ in general position. The set where the orbit dimension drops is given by the simultaneous vanishing of all determinants of $r \times r$-minors of the $q\times (\dim J^i)$-matrix defined by these $q$ prolonged vector fields. Note that even though this set in principle is easy to find, it can be quite complicated. In the following, we will skip most of the details regarding such computations and instead summarize the results. Computations were often done with the help Maple and, in particular, the PolynomialIdeals package  in addition to the DifferentialGeometry package.

In all the cases we consider, the set $\Sing^i$ can, for sufficiently large $i$, be chosen such that it has codimension 1. In this case, it can be given by a product of scalar relative invariants.  In particular, we have  $\Sing^i =\{R_1\cdots R_q=0\}$ for $i \geq m$ for some relative differential invariants $R_1, \dots, R_q$.
  \item  {\it For each of the relative invariants $R_1,\dots, R_q$ defining $\Sing^m$, consider the invariant underdetermined ODE $\Sigma_j^i \subset J^i$ given by $R_j=0$ and its total derivatives. Repeat steps (1) and (2) on each of these, i.e. find conditional absolute and relative invariants.} Any connected component of a determined system inside $\Sing^i$ must also be contained inside $\Sigma_j^i$ for some~$j$.
  \item The remaining determined systems, where $F$ and $G$ are necessarily of the same order $k=l$, are submanifolds of $\Sing^k$ of codimension 2 in $J^k$. In the cases we consider, they can be quickly singled out by analyzing the orbits on $J^i$ in detail for small~$i$.
\end{enumerate}

\subsection{Poincaré transformations on 3-dimensional Minkowski space} \label{sect:Minkowski}

Consider Lie algebra \eqref{L3} of vector fields preserving the flat metric, and denote it by $\mathfrak g$.
Real signature $(2,1)$ is equivalent to the Euclidean signature $(3,0)$ over $\C$; the former choice makes the zero locus of the first relative invariant $R_1$ real.
This and other lower order relative differential invariants are
 \begin{align*}
R_1 &=y_1^2+z_1^2-1, \\
R_2 &= (y_1^2-1)z_2^2-2y_1 z_1 y_2 z_2+(z_1^2-1)y_2^2, \\
R_{3a} &= z_2 y_3-y_2 z_3, \\
R_{3b} &= R_1 D_x(R_2)-3 D_x(R_1) R_2.
 \end{align*}
The weight $\lambda$ of any relative invariant satisfies\footnote{Here and below we use the basis $X_i$ as in the defining formula of $\g$, and indicate only the nontrivial $\lambda(X_i)$.}
$\lambda(X_4)=-w y_1$ and $\lambda(X_5)=-w z_1$ for some $w$, and is proportional the divergence of the prolonged vector fields with respect to the volume form $dx \wedge dy \wedge dz \wedge dy_1 \wedge dz_1$ on $J^1$.  In particular, the relative invariant $R_1$ has weight $w=2$, while $R_2$  and $R_{3a}$ have weights $w=6$, and $R_{3b}$ has weight $w=9$.

 \begin{theorem} \label{th:InvarMinkowski}
The field of absolute differential invariants is generated by the differential invariants
$I_2= \frac{R_2}{R_1^3}, I_{3a}=\frac{R_{3a}}{R_1^3}$ and the invariant derivation
$\nabla=\frac{R_{3b}}{R_1^5} D_x$.
 \end{theorem}

 \begin{proof}
It is straight-forward to verify that the indicated set is invariant,
and since the invariants are rational it is suffices to do this on infinitesimal level, that is
$X^{(3)}(I_2)=X^{(3)}(I_{3a})=0$, $[X^{(\infty)},\nabla]=0$ $\forall$ $X\in\g$.
We also compute that this set generates $2k-3$ algebraically independent rational invariants in each order $k\geq 2$.
To show that there exists no invariant field extension of rational differential invariants note that the invariants
$\nabla^{r}(I_2), \nabla^{r-1}(I_{3a})$ are linear when restricted to fibers of $J^{k+1}\to J^k$ for $k\geq 3$.
Thus it is enough to demonstrate the claim for the field of differential invariants of order 3, where it is a direct computation.
 \end{proof}

 \begin{remark}\label{rmk:derivation}
Solving the system $[X^{(\infty)},\nabla]=0$ for $\nabla=f D_x$ gives a simpler solution  $f= R_1^{-1/2}$, however such $\nabla$ is not invariant because the transformation $(x,y,z)\mapsto(-x,-y,z)$, which lies in the flow of $x\partial_y+y\partial_x$, maps $R_1^{-1/2} D_x$ to $-R_1^{-1/2} D_x$.

The generators in Theorem \ref{th:InvarMinkowski} are well-known. The invariant $I_2$ is the square of the curvature, $I_{3a}/I_2$ is proportional to the torsion, but differentiation by the natural parameter is not an invariant derivation, as we have just noted.
 \end{remark}

Now, let us make a detailed analysis of orbits on $J^k$ for small $k$. The action is transitive on $J^0$.
On $J^1$ there are two orbits: $\Sing^1=\{R_1=0\}$ and its complement.
On $J^2$ generic orbits have dimension 6 (codimension 1).
There are two 5-dimensional orbits, whose Zariski closures are $\{y_2=0,z_2=0\}$ and $\{R_1=0,D_x(R_1)=0\}$, respectively. The intersection of these two sets is the unique 4-dimensional orbit $\{R_1=0,D_x(R_1)=0,z_1 y_2-y_1 z_2=0\}$. All three orbits of dimension less than 6 lie inside the subset $\{R_2=0\} \subset J^2$, so we set $\Sing^2=\{R_2=0\}$. For $k >2$ all orbits lying outside $\Sing^k=\pi_{k,2}^{-1}(\Sing^2)$ are 6-dimensional. Any system of ODEs that is not strictly contained in $\Sing^{\infty}$ is given by  functions of rational absolute differential invariants.

 \begin{lemma}
Let $\E\subset J^k$ be a $\mathfrak g$-invariant determined ODE system of type \eqref{eq:GeneralSystem},
given by functions $F \in \mathcal{O}(J^k),G \in \mathcal{O}(J^l)$ of orders $k$ and $l\leq k$. Then, if $F$ and $G$ are not functions of rational absolute differential invariants, there are three possibilities:
\begin{itemize}
\item  $G=R_1$, $l=1$ and $k\geq 2$,
\item  $G=R_2$, $l=2$ and $k \geq 3$,
\item $F=y_2, G=z_2$ and $k=l=2$.
\end{itemize}
 \end{lemma}
Note that the third case is determined by looking for a subset $\{F=0,G=0\}$ which lies strictly inside $\{R_2=0\} \subset J^2$. Since the equation $\{F=0,G=0\} \subset J^2$ has codimension 2, it consists of orbits of dimension 5 or less. From the above analysis, we know there is only one determined system of this type: $\{y_2=0, z_2=0\}$.

The cases when $G=R_1$ and $G=R_2$ must been considered separately. Since the generators we found for absolute invariants vanish or diverge under these conditions, we must compute conditional differential invariants. Notice that we have $R_2=(y_2^2+z_2^2) R_1- (D_x(R_1)/2)^2$. This implies that the equation $R_2=0$ is a differential consequence of $R_1=0$ and, in particular,  does not give a determined system when added to the equation $R_1=0$.

Consider the subset
 \[
\Sigma^k=\{R_1=0,D_x(R_1)=0,\dots,D_x^{k-1}(R_1)=0\} \subset \pi_{k,1}^{-1}(\Sing^1) \subset J^k.
 \]
Since $R_1$ is a relative invariant, this set is preserved by $\mathfrak g$, and we consider the $\mathfrak g$-action on $\Sigma^k$. The action is transitive on $\Sigma^1$. On $\Sigma^2$, there is one singular orbit given by the vanishing of the conditional relative invariant $Q_2=z_1 y_2-y_1 z_2$. (On $\Sigma^3$ the vanishing of $Q_2$ is equivalent to the vanishing of $R_{3a}$.) This gives the  determined system $\{R_1=0,D_x(R_1)=0,Q_2=0\}$ which is exactly the 4-dimensional orbit in $J^2$ described above.  The $\mathfrak g$-action is transitive on $\Sigma^3\setminus \{Q_2=0\}$. It follows that orbits are 6-dimensional on $\Sigma^k \setminus \{Q_2=0\}$ for $k\geq 3$. The manifold $\Sigma^4 \subset J^4$ is 7-dimensional, so there is one algebraically independent conditional absolute differential invariant of order 4. Then there is one additional independent invariant of each higher order.
The first takes the form
 \[
J_4 = \frac{4 Q_2 D_x^2(Q_2)-7D_x(Q_2)+4 Q_2^4}{Q_2^3}.
 \]
We also have the conditional invariant derivation
 \[
\nabla_{\Sigma} = \frac{D_x(J_4)}{Q_2} D_x
 \]
 which, together with $J_4$, generates all conditional absolute differential invariants on $\Sigma^k$.

Next, we consider the subset
 \[
\Pi^k=\{R_2=0, D_x(R_2)=0,\dots, D_x^{k-2}(R_2)=0\} \subset J^k.
 \]
There is one 6-dimensional orbit on $\Pi^2$. The orbit dimension drops on the union of $\{y_2=0,z_2=0\}$ and
$\{R_1=0,D_x(R_1)=0\}$. The first of these contains a determined invariant system. On $\Pi^3$, we have
the following conditional absolute invariant and invariant derivation:
 \[
K_3=\frac{(R_1 y_3 -\frac{3}{2} D_x(R_1)) y_2^2}{(y_1^2 z_2-y_1 y_2 z_1-z_2)^3},\qquad  \nabla_{\Pi} = \frac{R_1 y_3 -\frac{3}{2} D_x(R_1)}{y_2 R_1^2} D_x.
 \]
They generate all conditional absolute invariants on $\Pi^k$. 

 \begin{remark} \label{remark:connected}
 The algebraic subsets $\Pi^k$ are reducible for $k \geq 3$. For example, $\Pi^3$ has 3 irreducible 7-dimensional components. Two of these components are $\{y_2=0,z_2=0\}$ and $\{R_1=0,D_x(R_1)=0, y_1^2 z_2-y_1 y_2 z_1-z_2=0\}$. The first of these is a determined ODE system, while all ODE systems lying inside the latter set was considered in the paragraph concerning $\Sigma^k$. Thus we are mainly concerned with the remaining irreducible component, and it is here that $K_3$ separates generic orbits.
 \end{remark}

 By collecting the results of the above computations, we obtain the following statement which gives a description of the invariant ODE systems.

 \begin{theorem} \label{th:invarEqMinkowski}
Let $\E\subset J^k$ be a $\mathfrak g$-invariant determined ODE system of type \eqref{eq:GeneralSystem},
given by functions $F \in \mathcal{O}(J^k),G \in \mathcal{O}(J^l)$ of orders $k$ and $l\leq k$. Then,
either $F$ and $G$ can be expressed through rational absolute differential invariants or
the system takes one of the following forms:
\begin{itemize}
 \item $G=R_1$ and either $F=Q_2$ or $F$ is a function of the conditional absolute invariants, which are generated by $\nabla_\Sigma$ and $J_4$.
 \item $G=R_2$ and $F$ is a function of the conditional absolute invariants, which are generated by $\nabla_{\Pi}$ and $K_3$.
 \item $G=y_2$ and $F=z_2$.
\end{itemize}
 \end{theorem}
Some of these differential equations have well-known geometric interpretations. The equation $R_1=0$ gives null curves, while the system $y_2=0,z_2=0$ gives geodesics in Minkowski space. The intersection of these, which can also be given by $R_1=0, Q_2=0$, has null geodesics as solutions. Finally, $R_2=0$ is the condition for vanishing of
curvature $\kappa$ of the curve ($R_2$ is the numerator of $\kappa^2$), but in non-Euclidean signature this does not
mean the curve is a straight line.

\subsection{Isometry algebra of a metric of constant nonzero curvature} \label{sect:so4}

Consider Lie algebra \eqref{L4} of Killing vectors for a space form of nonzero curvature.
Investigation of orbits for this Lie algebra, which we denote by $\mathfrak g$, is completely analogous to that of \eqref{L3}, therefore we omit the proofs.

We will express the absolute differential invariants through the following relative invariants:
\begin{align*}
R_1 &= e^x+4 y_1 z_1, \\
R_2 &= ( z_1-z_2) (y_1-y_2) R_1^2-(3 (y_1  z_1-y_1 z_2-y_2  z_1)^2+y_1^2  z_1^2-2 y_1^2 z_2^2-2  z_1^2 y_2^2) R_1 \\&\quad +4 y_1  z_1 (y_1  z_1-y_1 z_2-y_2  z_1)^2, \\
R_{3a} &= \left((z_1-z_2) y_3-(y_1-y_2) z_3+y_1 z_2-y_2 z_1\right) e^{2x} +\left((2 y_1 y_3-3 y_2^2) z_1^2-(2 z_1 z_3-3 z_2^2) y_1^2\right) e^x, \\
R_{3b} &= e^{x/2} (R_1 D_x(R_2)-3 D_x(R_1) R_2).
\end{align*}
The weight $\lambda$ of each of these relative invariants satisfies
\[\lambda(X_3)=w,\qquad  \lambda(X_5)=w(y_0-2y_1),\qquad  \lambda(X_6)=w(z_0-2 z_1) ,\]
for some integer $w$. For $R_1,R_2,R_{3a}, R_{3b}$, we have $w=2,w=6,w=6, w=9$, respectively.

\begin{theorem} \label{th:InvarSpaceform}
The field of absolute differential invariants is generated by the differential invariants $I_2 = \frac{R_2}{R_1^3}$, $I_{3a} = \frac{R_{3a}}{R_1^3}$ and the invariant derivation $\nabla= \frac{e^{x/2} R_{3b}}{R_1^5} D_x$.
\end{theorem}


The description of orbits on $J^k$ for small $k$ follows that of the previous Lie algebra, step by step.\footnote{This can be explained by the fact that both Lie algebras preserve a metric of constant curvature. Due to transitivity of $\mathfrak g$ on $J^0$, the differential invariants of order $k$ are in one-to-one correspondence with the invariants of the prolonged action on $J^k$ of the stabilizer of a
point in $J^0$, which is $\mathfrak{so}(3)$ in both cases.} The action is transitive on $J^0$. On $J^1$, there are two orbits: $\Sing^1=\{R_1=0\}$ and its complement.  On $J^2$ there is one 4-dimensional orbit, namely $\{R_1=0,D_x(R_1)=0,z_1 y_2-y_1 z_2=0\}$, and two 5-dimensional orbits. The Zariski closures of the latter ones are
\[\{y_2-y_1 (1+2 e^{-x} y_1 z_1 )=0,z_2-z_1 (1+2 e^{-x} y_1 z_1 )=0\} \text{  and  } \{R_1=0,D_x(R_1)=0\}.\] The intersection of these two sets is the unique 4-dimensional orbit. The orbits in general position are 6-dimensional. All three orbits of dimension less than 6 lie inside the subset $\{R_2=0\} \subset J^2$. Set $\Sing^2=\{R_1 R_2=0\}$. For  $k >2$, all orbits lying outside $\Sing^k=\pi_{k,2}^{-1}(\Sing^2)$ are 6-dimensional. 

Consider the $\mathfrak g$-invariant subset
\[ \Sigma^k = \{R_1=0,D_x(R_1)=0,\dots,D_x^{k-1}(R_1)=0\} \subset J^k.\]
There is a conditional relative invariant $Q_2= z_1 y_2-y_1 z_2$. 
The vanishing of $Q_2$ on $\Sigma^2$ is equivalent to the vanishing of $R_{3a}|_{\Sigma^2}$ and gives exactly the set on $\Sigma^2$ where the orbit dimension drops (down to 4). The $\mathfrak g$-action is transitive on $\Sigma^3\setminus \{Q_2\}$. Thus orbits are 6-dimensional on $\Sigma^k \setminus\{Q_2\}$ for $k \geq 3$.

The conditional absolute invariants on $\Sigma^k$ are generated by
\[J_4= \frac{e^x \left(4 Q_2 D_x^2(Q_2)-7 D_x(Q_2)^2+2Q_2 D_x(Q_2) \right)}{Q_2^3} -16 e^{-x} Q_2 \]
and the invariant derivation
\[\nabla_\Sigma 
=\frac{D_x(J_4)}{Q_2} e^x D_x.\]

Consider now the $\mathfrak g$-invariant underdetermined ODE
\[ \Pi^k = \{R_2=0, D_x(R_2)=0, \dots , D_x^{k-2} (R_2)=0\} \subset J^k .\]
The orbit dimension on $\Pi^2$ drops on the union of the subset \[
\{y_2-y_1 (1+2 e^{-x} y_1 z_1 )=0,z_2-z_1 (1+2 e^{-x} y_1 z_1 )=0\} \cup \{R_1=0,D_x(R_1)=0\},\]
 and the complement of this subset is a 6-dimensional orbit. The first of these components is a determined system, while the second component is the intersection of $\Pi^2$ with $\Sigma^2$.

We have a third-order conditional absolute differential invariant
\begin{gather*}
K_3= \frac{1}{(4 y_1 z_1+e^{x}) (2 y_1^2 z_1+e^{x} (y_1-y_2))^2} \Big( (y_1-y_2) (y_1-3 y_2+2 y_3)  e^{3x} +\big((8 z_1-4 z_2) y_1^3 \\
+(4(3 z_2-5 z_1) y_2+4 y_3 (z_1+z_2)) y_1^2-4(3 y_2^2 z_2+ y_2 y_3 z_1) y_1+12 y_2^3 z_1\big)  e^{2x}+12 y_1^3 z_1^2 (y_1-2 y_2)  e^{x} \Big)
\end{gather*}
and the conditional invariant derivation $\nabla_{\Pi}=\frac{e^x}{R_1} D_x(K_3) D_x$. These generate the conditional differential invariants. Notice that $\Pi^k$ is in general reducible. But similar to what was explained in  Remark \ref{remark:connected}, most of the ODE systems contained in the
irreducible components have already been considered in the discussion concerning $\Sigma^k$, and the remaining ones are given by the conditional differential invariants.

%

 \begin{theorem} \label{th:invarEqSpaceform}
Let $\E\subset J^k$ be a $\mathfrak g$-invariant determined ODE system of type \eqref{eq:GeneralSystem},
given by functions $F \in \mathcal{O}(J^k),G \in \mathcal{O}(J^l)$ of orders $k$ and $l\leq k$. Then,
either $F$ and $G$ can be expressed through rational absolute differential invariants or
the system takes one of the following forms:
\begin{itemize}
 \item $G=R_1$ and either $F=Q_2$ or $F$ is a function of the conditional absolute invariants which are generated by $\nabla_\Sigma$ and $J_4$.
 \item $G=R_2$ and $F$ is a function of the conditional absolute invariants which are generated by $\nabla_{\Pi}$ and $K_3$.
 \item $G=y_2-y_1 (1+2 e^{-x} y_1 z_1)$ and $F=z_2-z_1 (1+2 e^{-x} y_1 z_1)$.
\end{itemize}
 \end{theorem}

Note that the last system $\{G=0,F=0\}$ is point-equivalent to $\{y_2=0,z_2=0\}$. As in \S\ref{sect:Minkowski} solutions of this system are the geodesics of the metric preserved by $\mathfrak g$. Again, $R_1=0$ gives exactly the null curves. The solutions of the determined system $R_1=0,Q_2=0$ are null geodesics. And $R_2=0$ expresses zero square curvature of a curve in a space form.

\subsection{A 3D action of $\mathfrak{sp}(4)$}  \label{sect:sp4}

 Let now $\mathfrak g$ denote Lie algebra \eqref{L5}. The first relative invariants $\mathfrak g$ are the following:
\begin{align*}
R_1&= y z_1 - z y_1 +1 \\
R_3 &=z_2 y_3- y_2 z_3\\
R_4 &= 3 R_1^2 D_x(R_3)^2+4 R_1 R_3 (4 R_1 ( y_3 z_4-  z_3 y_4) +3 D_x(R_1) D_x(R_3)) \\&-4 R_3^2 \left(16 R_1 (y z_3-z y_3)-15 D_x(R_1)^2+24 R_1 ( y_1 z_2- z_1 y_2 )\right)\\
R_{5a} &= 2 R_1 R_3 (z_2 y_5-y_2 z_5)-3 R_1 D_x(R_3)^2 +4 R_3^2 (y z_3-z  y_3+z_1 y_2-y_1 z_2) \\&+2 R_3 \left(3 R_1 (z_3 y_4-y_3 z_4)-D_x(R_1) D_x(R_3)\right) \\
R_{5b} &= R_1 R_3 D_x(R_4)-\left(R_3 D_x(R_1)+\tfrac{5}{2} R_1 D_x(R_3)\right)R_4.
\end{align*}
The weight lattice in this case is two-dimensional, and given by
 \begin{gather*}
\lambda(X_2)=w_1 z_1, \qquad \lambda(X_3)=-w_1 y_1, \qquad \lambda(X_7)=-w_2, \\
\lambda(X_8)=(w_1-w_2) y-w_1 xy_1, \qquad \lambda(X_9)=(w_1-w_2) z-w_1 x z_1, \qquad \lambda(X_{10})=-w_2 x.
 \end{gather*}
For invariants $R_1, R_3, R_4, R_{5a}, R_{5b}$ the weights $(w_1,w_2)$ are $(1,0), (6,8), (16,20), (15,20), (24,30)$.

 \begin{theorem} \label{th:Invarsp4}
The field of absolute differential invariants is generated by $I_4=\frac{R_4^2}{R_1^2 R_3^{5}},I_{5a}=\frac{R_{4} R_{5a}}{R_1 R_3^{5}}$ and the invariant derivation $\nabla=\frac{R_{5b}}{R_1 R_3^4} D_x$.
 \end{theorem}

 \begin{remark}
Although the derivation $R_1^{1/2} R_3^{-1/4} D_x$ commutes with the prolonged action of $\mathfrak g$, it is not invariant since
it changes sign under the transformation $(x,y,z)\mapsto(-x,-y,z)$. The multivalued functions $\frac{R_4}{R_1 R_3^{5/2}}$ and $\frac{R_{5a}}{R_3^{5/2}}$ are infinitesimally invariant, but they change sign under the transformation $(x,y,z) \mapsto (ix,iy,z)$. Both of these transformations belong to the flow of $x \partial_x+y\partial_y$, which is contained in our Lie algebra of vector fields.
 \end{remark}

We analyse the orbits on $J^k$ for low $k$. On $J^1$, the orbits are $\Sing^1=\{R_1=0\}$ and its complement.
On $J^2$, there is one 7-dimensional orbit. The set on which the orbit-dimension drops is $\Sing^2 = \{R_1=0\} \cup \{y_2=0,z_2=0\}$. The set on which the orbit-dimension is $\leq 5$ is $\{R_1=0,D_x(R_1)=0\} \cup \{y_2=0,z_2=0\}$.  The two irreducible components of $\Sing^2$ intersect in the unique 4-dimensional orbit $\{R_1=0, D_x(R_1)=0, z_1 y_2-y_1 z_2=0\}$.
Thus we have two determined systems on $J^2$, similar to the previous cases.

On $J^3$ there is one 9-dimensional orbit. The set on which the orbit-dimension drops is $\Sing^3=\{R_1 R_3=0\}$.
The radical polynomial ideal of the set of orbits with $\dim\leq7$ contains the element $z_2 R_1$, hence
there are no determined systems of codimension 2 in $J^3$.

The space $J^4$ is 11-dimensional, and generic orbits are 10-dimensional. The set of points where the orbit dimension is less than or equal to 9 is not easy to describe precisely. However, it can be checked that its radical ideal contains the polynomial $R_1 R_3$. Thus we set $\Sing^4= \pi_{4,3}^{-1}(\Sing^3)$, and in general $\Sing^{k}=\pi^{-1}_{k,3}(\Sing^3)$ for $k \geq 4$. All orbits outside $\Sing^k$ are 10-dimensional for $k \geq 4$.

Let $\Sigma^k$ denote the underdetermined system given by $R_1=0$:
\[ \Sigma^k= \{R_1=0,D_x(R_1)=0,\dots,D_x^{k-1}(R_1)=0\} \subset \pi_{k,1}^{-1}(\Sing^1) \subset J^k.\]
The $\mathfrak g$-action is transitive on $\Sigma^1$. On $\Sigma^2$ there is one 5-dimensional orbit and one 4-dimensional orbit. The latter is given by the vanishing of $Q_2=z_1 y_2 - y_1 z_2$. (On $\Sigma^3$, the vanishing of $Q_2$ is equivalent to that of $R_{3}|_{\Sigma^3}$.) The set $\Sigma^3$ consists of one 6-dimensional orbit, one 5-dimensional orbit, whose Zariski closure is given by $Q_2=0$, and one 4-dimensional orbit given by $Q_2=0, D_x(Q_2)=0$.

On $\Sigma^4$ there is no invariant submanifold of codimension 1 given by additional constraints of order 4. The same goes for $\Sigma^5$. On $\Sigma^6$, there is one 9-dimensional orbit. The orbit dimension drops exactly on the set $Q_2 Q_6=0$, where $Q_6$ is given by
 \begin{align*}
 Q_6 &=10 Q_2^3 D_x^4(Q_2)-70 Q_2^2 D_x(Q_2) D_x^3(Q_2) +\left(16 Q_2^4+280 Q_2 D_x(Q_2)^2\right) D_x^2(Q_2) \\& -49 Q_2^2 D_x^2(Q_2)^2-175 D_x(Q_2)^4-20 Q_2^3 D_x(Q_2)^2+9 Q_2^6.
 \end{align*}
%

 On $\Sigma^7$ the complement of (the Zariski closure of) $\{Q_2 Q_6=0\}$ is a 10-dimensional orbit, and on $\Sigma^k$ for $k\geq 8$ the complement of $\{Q_2 Q_6=0\}$ consists of 10-dimensional orbits. The conditional absolute invariants are generated by
 \[J_8=\frac{Q_8^2}{Q_6^5}, \qquad \nabla_{\Sigma}= \frac{Q_2^2 \left( 2 Q_6 D_x( Q_8)-5 Q_8 D_x(Q_6)\right)}{Q_6^4} D_x,\]
 where $Q_8$ is given by
\begin{align*}
Q_8 &= 40 Q_2^2 Q_6 D_x^2(Q_6)-45 Q_2^2 D_x(Q_6)^2+40 Q_2 D_x(Q_2) Q_6 D_x(Q_6) \\
&-\left(32 Q_2^3+224 Q_2 D_x^2(Q_2)-160 D_x(Q_2)^2\right) Q_6^2.
\end{align*}

The absolute invariant
\[\nabla_{\Sigma}(J_8) = \frac{Q_2^2 Q_8 \left(2 Q_6 D_x(Q_8)-5Q_8D_x(Q_6)\right)^2}{Q_6^{10}}\]
is of degree 2 when restricted to a fiber of $J^9 \to J^8$. Due to the factor $Q_8$, there is no algebraic extension of the field
$\langle J_8,\nabla_{\Sigma}(J_8)\rangle$ of rational conditional invariants that contains an invariant of lower degree on this fiber. On the other hand, $\nabla_\Sigma(\nabla_\Sigma(J_8))$ is linear on the fibers of $J^{10}\to J^9$. Thus, there is no nontrivial extension of the field generated by $J_8$ and $\nabla_\Sigma$ inside the field of rational conditional absolute invariants.

Now consider the $\mathfrak g$-action on
 \[
\Pi^k= \{R_3=0, D_x(R_3)=0,\dots, D_x^{k-3}(R_3)=0\} \subset \pi_{k,3}^{-1}(\Sing^3) \subset J^k.
 \]
There is one 8-dimensional orbit on the complement of $\{y_2=0,z_2=0\} \cup \{R_1=0,R_3=0\}$ in $\Pi^3$.
The set $\Pi^4$ consists of two irreducible 9-dimensional components:
\[ \{y_2=0, z_2=0\}, \qquad \{R_3=0, D_x(R_3)=0, z_3 y_4-y_3 z_4=0\}.\]
The first of these was discussed above. In order to avoid this component (and higher-order analogues), we use coordinates $x,y,y_1,\dots, y_k,z, z_1, z_2$ on $\Pi^k$ and replace in our formulas all occurrences of $z_i$ for $i \geq 3$ using $z_3=y_3 z_2/y_2$ and its derivatives.

There is a conditional relative invariant of order 4:
 \begin{align*}
 P_4  &= (3 z_2 z_4-4 z_3^2) R_1^2-6 z_2 \left(z_3 D_x(R_1)+3 z_2 (y_1 z_2-y_2 z_1)\right) R_1+9 z_2^2 D_x(R_1)^2.
 \end{align*}

Orbits are 9-dimensional on the complement of $P_4=0$ in $\Pi^4\setminus\{y_2=0\}$. Also on $\Pi^5\setminus\{y_2=0\}$, the orbits are 9-dimensional on the complement of $P_4$. The two conditional absolute invariants on $\Pi^k$ are generated by the conditional absolute invariants
\begin{align*}
 K_5 &= \frac{z_2^6 R_1^6 (9 y_2^2 y_5-45 y_2 y_3 y_4+40 y_3^3)^2}{y_2^6 P_4^3}, \\
 K_6 &= \frac{z_2^4 R_1^3 }{y_2^4 P_4^2} \Big((9 y_2^3 y_6-72 y_2^2 y_3 y_5-45 y_2^2 y_4^2+300 y_2 y_3^2 y_4-200 y_3^4) R_1 \\
 &+(27 y_2^3 y_5-135 y_2^2 y_3 y_4+120 y_2 y_3^3) D_x(R_1)  \Big),
 \end{align*}
 and the conditional invariant derivation
 \[\nabla_\Pi= D_x(K_5)^{-1} D_x.\]

 \begin{theorem}\label{tm24}
Let $\E\subset J^k$ be a $\mathfrak g$-invariant determined ODE system as in \eqref{eq:GeneralSystem},
given by functions $F \in \mathcal{O}(J^k),G \in \mathcal{O}(J^l)$ of orders $k$ and $l\leq k$. Then,
either $F$ and $G$ can be expressed through rational absolute differential invariants or
the system takes one of the following forms:
 \begin{itemize}
 \item $G=R_1$ and  $F=Q_2$, $F=Q_6$ or $F$ is a function of the conditional absolute invariants which are generated by $J_8$ and $\nabla_\Sigma$.
 \item $G=R_3$ and  $F=P_4$ or $F$ is a function of the conditional absolute invariants which are generated by $K_5$, $K_6$ and $\nabla_{\Pi}$.
 \item $G=y_2$ and $F=z_2$.
 \end{itemize}
 \end{theorem}

The equation $R_1=0$ describes Legendrian curves, while $R_3=0$ describes curves lying in a plane;
and the system $\{y_2=0,z_2=0\}$ gives indeed all straight lines in $\C^3$.

\subsection{Foliation-preserving algebra Lie{\hskip1pt}16} \label{sect:Lie16}

For Lie algebra $\mathfrak g$ defined by \eqref{L6}, the functions
 \[
R_1=z_1-y, \qquad R_2=y_2, \qquad R_{3a}=(z_2-y_1) y_3-y_2 z_3
 \]
are relative differential invariants. The weight $\lambda$ of any of these is of the form
 \[
\lambda(X_4)=-w, \qquad \lambda(X_5)=-w y_1,
 \]
with $w$ equal to $1$, $3$ and $6$, respectively.
By computations similar to above we get:

 \begin{theorem} \label{th:Invar16}
The field of rational absolute differential invariants is generated by
$I_2=R_2/R_1^3$, $I_{3a}=R_{3a}/R_1^6$ and $\nabla=R_1^{-1} D_x$.
 \end{theorem}



The $\mathfrak g$-action is transitive on $J^0$. On $J^1$ there is one open 5-dimensional orbit and one 4-dimensional orbit
given by $\Sing^1= \{R_1=0\}$. On $J^2$ orbits in general position are 6-dimensional. The orbit-dimension drops on the set
$\Sing^2= \{R_1R_2=0\}$. Its irreducible components intersect in the 4-dimensional orbit $\{R_1=0, D_x(R_1)=0,R_2=0\}$.
It is clear that $\g$-orbits on $J^k$ are 6-dimensional outside the subset $\Sing^k=\pi_{k,2}^{-1}(\Sing^2)$.

Let us restrict to the underdetermined ODE
 \[
\Sigma^k = \{R_1=0, D_x(R_1)=0, \dots, D_x^{k-1}(R_1)=0\} \subset J^k .
 \]
The action is transitive on $\Sigma^1$. On $\Sigma^2$ there is one open 5-dimensional orbit and one 4-dimensional orbit,
separated by the vanishing of $R_2$. This gives the first invariant determined system: $\{R_1=0,D_x(R_1)=0, R_2=0\}$.
On $\Sigma^3$ there is one 6-dimensional orbit, one 5-dimensional orbit (whose Zariski closure is given by $R_2=0$) and one 4-dimensional orbit, given by $R_2=0$ and $D_x(R_2)=0$. On $\Sigma^4$ generic orbits are 6-dimensional, and the orbit dimension drops on the set given by $R_2=0$.
The conditional absolute invariants on $\Sigma^k$ are generated by
 \[
J_4=\frac{(3 y_2 y_4-5 y_3^2)^3}{y_2^{8}}, 
\qquad \nabla_{\Sigma} = \frac{3 y_2 y_4-5y_3^2}{y_2^3} D_x.
 \]

Notice that $J_4$ is equal to the restriction of
${(3I_2 (\nabla^2(I_2)-3I_{3a})-5\nabla(I_2)^2-9I_2^3)^3}/{I_2^8}$ to $\Sigma^4$.
Thus this can be viewed as an actual absolute invariant, not only conditional. One can also rewrite  $\nabla_{\Sigma}$ in this way,
but then one needs to use $R_1$ in addition. 

Now we restrict to
 \[
\Pi^k = \{R_2=0, D_x(R_2)=0, \dots, D_x^{k-2} (R_2)=0 \} \subset J^k.
 \]
Generic orbits in $\Pi^k$ are 5-dimensional for every $k$, and the dimension drops only when both $R_1=0$ and $D_x(R_1)=0$. The conditional absolute invariants on $\Pi^k$ are generated by $\nabla$ and
\[K_2=\frac{D_x(R_1)}{R_1^2}.\]

 \begin{theorem} \label{th:invarEq16}
Let $\E\subset J^k$ be a $\mathfrak g$-invariant determined ODE system as in \eqref{eq:GeneralSystem},
given by functions $F \in \mathcal{O}(J^k),G \in \mathcal{O}(J^l)$ of orders $k$ and $l\leq k$. Then,
either $F$ and $G$ can be expressed through rational absolute differential invariants or
the system takes one of the following forms:
 \begin{itemize}
\item $G=R_1$ and either $F=R_2$ or $F$ is a function of the conditional absolute invariants generated by $J_4$ and $\nabla_{\Sigma}$. 
\item $G=R_2$ and $F$ is a function of the conditional absolute invariants which are generated by $K_2$ and  $\nabla$.
 \end{itemize}
 \end{theorem}
\begin{remark}\label{rk:point}
A consequence of choosing coordinates so that the 1-dimensional invariant foliation (given by $x=\text{const},y=\text{const}$) is tangent to the fibers of $J^0$ in our split coordinates (given by $x=\text{const}$) is that it can not be given by equations of the form $y=\tilde y(x), z=\tilde z(x)$. However, after applying the point transformation $(x,y,z) \mapsto (z,y,x)$, the 1-dimensional leaves can be given by $y=\text{const}$ and $z=\text{const}$ and they are thus solutions to the system $y_1=0,z_1=0$. Thus, in addition to the systems described in the theorem above, there is one additional determined system of order 1 if the 1-dimensional foliation is not tangent to the fibers $x=\text{const}$. This phenomenon occurs also for the next two Lie algebras we consider (Lie{\hskip1pt}27 preserves two 1-dimensional foliations, one of which is not tangent to $x=\text{const}$). See also the discussion in \S\ref{sl2}.
\end{remark}


\subsection{Foliation-preserving algebra Lie{\hskip1pt}27}
\label{sect:Lie27}

Consider Lie algebra \eqref{L7} and denote it by $\mathfrak g$. The following are the first relative differential invariants:
 \begin{align*}
R_1 =\ & y_1-z,\\
R_{2a} =\ & y_2, \\
R_{2b} =\ & z_1 y_2+(z-y_1) z_2-2 z_1^2, \\
R_3 =\ & R_1 \left( R_{2b} D_x(R_{2a})- R_{2a} D_x(R_{2b})\right) +3 R_{2a} R_{2b} (D_x(R_1) -R_{2a})), \\
R_{4a} = \ & R_1^2 \left(3 R_{2a} D_x^2(R_{2a})-4 D_x(R_{2a})^2 \right)-3 R_{2a}^2 \left(3 R_{2a}^2-6 R_{2a} D_x(R_1)+2 R_1D_x(R_{2a})\right),\\
R_{4b} = \ & R_1 R_{2a} D_x(R_3)- R_1^2 R_{2a} R_{2b} D_x^2(R_{2a}) -\left(\tfrac{8}{3} R_1 D_x(R_{2a})+4 D_x(R_1) R_{2a}-4 R_{2a}^2\right) R_3 \\
&+ \tfrac{4}{3} R_1^2 R_{2b} D_x(R_{2a})^2+ 2 R_1 R_{2a}^2 R_{2b} D_x(R_{2a}) -3 R_{2a}^3 R_{2b}  \left( 2 D_x(R_1)-R_{2a} \right).
 \end{align*}
 The weight $\lambda$ of these relative invariants satisfies
 \begin{gather*}
 \lambda(X_4)=3 w_1-2 w_2, \quad \lambda(X_5)=(3 w_1-2 w_2+w_3) z-w_3 y_1, \quad
  \lambda(X_6)=-w_1, \\ \lambda(X_7)=-w_2 x, \quad \lambda(X_8)=((3 w_1-2 w_2+w_3)z-w_3 y_1) x-(3 w_1-w_2) y
 \end{gather*}
 where $(w_1,w_2,w_3)$ are equal to  $(0,1,1)$, $(1,3,3)$, $(2,6,3)$, $(4,12,8)$, $(4,12,10)$ and $(6,18,13)$, respectively.

 \begin{theorem} \label{th:Invar27}
The field of rational differential invariants is generated by the invariants $I_{3}=\frac{R_3^3}{(R_{2a} R_{2b})^4}$,
$I_{4a}=\frac{R_{4a}^2}{R_3 R_{2a}^4}$, $I_{4b}=\frac{R_{4b}^2R_{2a}^8}{R_{4a}^5}$
and the invariant derivation $\nabla=\frac{R_1 R_{2a} R_{2b}}{R_3} D_x$.
 \end{theorem}


The action is transitive on $J^0$. The space $J^1$ consists of one 5-dimensional orbit.
The submanifold $\Sing^1=\{R_1=0\}$ is exactly the set of points where the orbit dimension drops.
On $J^2$ there is one open 7-dimensional orbit.
The orbits of lower dimension unite in the subset $\Sing^2=\{R_1 R_{2a} R_{2b}=0\}$.
It can easily be verified that the subset in $J^3$ where the orbit dimension drops is contained in $\Sing^3=\pi_{3,2}^{-1}(\Sing^2)$,
and it follows that orbits in $J^k$ for $k\geq 3$ are 8-dimensional outside the subset $\Sing^k=\pi_{k,2}^{-1}(\Sing^2)$.

Now let us restrict to the subset
 \[
\Sigma^k=\{R_1=0, D_x(R_1)=0, \dots, D_x^{k-1}(R_1)=0\} \subset J^k.
 \]
On $\Sigma^1$ there are two orbits, separated by the vanishing of the conditional relative invariant $Q_1= z_1$.
In $\Sigma^2$ the complement of $Q_1=0$ consists of a 5-dimensional orbit. Notice that $R_{2a}|_{\Sigma^2}=Q_1$.
In $\Sigma^3$ the complement of $\{Q_1=0\}$ consists of a 6-dimensional orbit. On $\Sigma^4$, we have a new conditional
relative invariant $Q_4=9 z_1^2 z_4-45 z_1 z_2 z_3+40 z_2^3$. The complement of $\{Q_1 Q_4=0\}$ is a 7-dimensional orbit.
In $\Sigma^5$, the complement of $\{Q_1 Q_4=0\}$ is an 8-dimensional orbit. In $\Sigma^k$ for $k \geq 6$, the complement of $\{Q_1 Q_4=0\}$
consists of 8-dimensional orbits. Using the notation
\[Q_6=2 Q_1 Q_4  \left(Q_1 D_x^2(Q_4) +D_x(Q_1) D_x(Q_4) \right)-\tfrac{7}{3} Q_1^2 D_x(Q_4)^2-\left(9 Q_1D_x^2(Q_1)-7 D_x(Q_1)^2\right) Q_4^2,\]
we get the following generators for the conditional absolute invariants:
 \begin{align*}
J_6 = \frac{Q_6^3}{Q_4^8}, \qquad \nabla_{\Sigma} = \frac{Q_1 Q_6}{Q_4^3} D_x.
 \end{align*}

Next, consider the subset
 \[
\Pi_a^k = \{R_{2a}=0, D_x(R_{2a})=0, \dots, D_x^{k-2}(R_{2a})=0\} \subset J^k.
 \]
On $\Pi_a^2$ there is an open 6-dimensional orbit. The orbit dimension drops on the subset $\{R_1 R_{2b}=0\}$.
On $\Pi_a^3$ the open orbit is 7-dimensional, and the orbit dimension drops on the subset of $\{R_1 R_{2b}=0\}$.
On $\Pi_a^4$ there is a conditional relative invariant:
 \begin{align*}
P_4 &= 3 R_{2b} z_4+4 R_1 z_3^2+6 z_2 \left(4 z_1 z_3-3 z_2^2\right).
 \end{align*}
The complement of $\{R_1 R_{2b} P_4=0\}$ is an 8-dimensional open orbit. Furthermore, each orbit on the complement of $\{R_1 R_{2b} P_4=0\}$ in $\Pi_a^k$ is 8-dimensional, for every $k\geq 4$. The following is a conditional relative invariant on $\Pi_a^5$:
\begin{align*}
P_5 &= \left(9 z_2^2 z_5-45 z_2 z_3 z_4+40 z_3^3\right) R_1^3+\left(36 z_2 z_1 (z_1 z_5-5 z_2 z_4+10 z_3^2)-90 z_3 (z_1^2 z_4+z_2^3)\right) R_1^2
\\&+\left(18 z_1^3 (2 z_1 z_5-25 z_2 z_4)+90 z_2^2 z_1 (14 z_1 z_3-9 z_2^2)\right) R_1-180 z_1^3 \left(z_1^2 z_4-4 z_1 z_2 z_3+3 z_2^3\right)
\end{align*}
 The conditional absolute invariants are generated by
 \begin{align*}
K_5 = \frac{P_5^2}{(R_1 P_4)^3}, \qquad  \nabla_a = \frac{R_{2b} P_5}{(R_1 P_4)^2} D_x.
 \end{align*}

Consider now the subset
 \[
\Pi_b^k = \{R_{2b}=0, D_x(R_{2b})=0, \dots, D_x^{k-2}(R_{2b})=0\} \subset J^k.
 \]
On $\Pi_b^2$ there is an open 6-dimensional orbit. The orbit dimension drops on the subset $\{R_1 R_{2a}=0\}$. All ODE systems in the subset given by $\{R_1 R_{2a}=0\}$ are already considered.

The set $\Pi_b^3$ is reducible and consists of two 7-dimensional irreducible components:
\[ \{R_1=0, z_1=0\} , \qquad \{R_{2a}=0, D_x(R_{2a}) =0, z_2 y_3-y_2 z_3+2 z_1 z_3-3 z_2^2=0\}. \]
The first component is defined by equations of lower order and hence was already treated. In order to avoid this subset (and its higher-order analogues), we use coordinates $x,y,y_1,z,z_1,\dots,z_k$ on $\Pi_b^k$ by replacing $y_i$ for $i \geq 2$ using the formula $y_2=(2 z_1^2+(y_1-z) z_2)/z_1$ and its derivatives.

On $\Pi_b^3\setminus\{z_1=0\}$ the open orbit is 7-dimensional, and the orbit dimension drops on the subset $\{R_1 R_{2a}=0\}$.
On $\Pi_b^4\setminus\{z_1=0\}$ there is a conditional relative invariant:
 \[
T_4= 3 z_1^2 (R_1 z_2+2 z_1^2) z_4-4 R_1 z_1^2 z_3^2-6 z_1 z_2 (R_1 z_2+6 z_1^2) z_3+9 z_2^3 (R_1 z_2+4 z_1^2).
 \]
The complement of $R_1 R_{2a} T_4=0$ is an 8-dimensional open orbit. Furthermore, each orbit on the complement of $R_1 R_{2a} T_4=0$ in $\Pi_b^k \setminus\{z_1=0\}$ for every $k\geq 4$ is 8-dimensional. If we let
 \begin{equation*}
 T_5 = 3 R_1\left(D_x(R_1)+z_1\right) z_1^2 D_x(T_4)-z_1^2\left(10 R_1 D_x^2(R_1)+3 D_x(R_1)^2-5z_1 D_x(R_1)-28z_1^2\right) T_4,
 \end{equation*}
 then the conditional absolute invariants are generated by
 \begin{align*}
 L_5=\frac{T_5^2}{(R_1 T_4)^3}, \qquad  \nabla_b = \frac{z_1^2 R_{2a} T_5}{(R_1 T_4)^{2}} D_x.
 \end{align*}

 \begin{theorem} \label{th:invarEq27}
Let $\E\subset J^k$ be a $\mathfrak g$-invariant determined ODE system as in \eqref{eq:GeneralSystem},
given by functions $F \in \mathcal{O}(J^k),G \in \mathcal{O}(J^l)$ of orders $k$ and $l\leq k$. Then,
either $F$ and $G$ can be expressed through rational absolute differential invariants or
the system takes one of the following forms:
 \begin{itemize}
\item $G=R_1$ and $F=Q_1$, $F=Q_4$ or $F$ is a function of the conditional absolute invariants which are generated by $J_6$ and  $\nabla_\Sigma$.
\item $G=R_{2a}$ and $F=R_{2b}$, $F=P_4$ or $F$ is a function of the conditional absolute invariants which are generated by $K_5$ and $\nabla_a$.
\item $G=R_{2b}$ and $F=R_{2a}$, $F=T_4$ or $F$ is a function of the conditional absolute invariants which are generated by $L_5$ and $\nabla_b$.
 \end{itemize}
 \end{theorem}

\subsection{Foliation-preserving algebra Lie{\hskip1pt}29} \label{sect:Lie29}

For Lie algebra \eqref{L8}, which we in this section denote by $\mathfrak g$, the first relative differential invariants are
 \begin{align*}
 R_2 &= y_2, \\
R_3 &= (3 (2 z_1^2+3 z_2) y_3-9 y_2 z_3-2 z_1 y_2 (2 z_1^2+9 z_2)) e^{2z}, \\
R_{4a} &= (3 y_2 y_4-5 y_3^2-4 z_1 y_2 y_3 +4 y_2^2( z_1^2+3  z_2) ) e^{4z/3}.
 \end{align*}
 The weight $\lambda$ of these relative invariants satisfies
 \begin{gather*}
   \lambda(X_4)=-w, \quad \lambda(X_5)=- w y_1, \quad \lambda(X_6)=-w/3, \quad \lambda(X_7)=-w x, \quad \lambda(X_8)=-w x y_1,
 \end{gather*}
 where $w$ is equal to $3$, $6$ and $8$, respectively.

 \begin{theorem} \label{th:Invar29}
The field of absolute differential invariants is generated by the differential invariants $I_3= \frac{R_3}{R_2^2}$, $I_{4a} = \frac{R_{4a}^3}{R_2^8}$ and the invariant derivation $\nabla =e^{2z/3} \frac{R_{4a}}{R_2^3}  D_x$.
 \end{theorem}

The Lie algebra $\mathfrak g$ is transitive on $J^1$. There is only one orbit. On $J^2$ there is one 7-dimensional orbit.
Its complement $\Sing^2=\{R_2=0\}$ consists of orbits of lower dimension. On $J^3$ generic orbits are 8-dimensional.
The subset on which the orbit dimension drops is $\pi_{3,2}^{-1}(\Sing^2)$.
In a similar way, the orbits on $J^k$ are 8-dimensional outside the set $\pi_{k,2}^{-1}(\Sing^2)$ for $k > 3$.

Thus, the invariant determined systems that are not given by absolute invariants lie inside the subset
 \[
\Sigma^k = \{ R_2=0, D_x(R_2)=0, \dots, D_x^{k-2}(R_2)=0 \} \subset J^k.
 \]
On $\Sigma^2$, there is one 6-dimensional orbit and one 5-dimensional orbit. The latter is given by the additional constraint $Q_2=3z_2+2 z_1^2$. Also for $k>2$ the generic orbits in $\Sigma^k$ are 6-dimensional in the complement of $\{Q_2=0\}$. The conditional absolute invariants are generated by
 \begin{align*}
J_3 &= \frac{(9z_3+36 z_1 z_2+16 z_1^3)^2}{Q_2^3}, \quad
\nabla_{\Sigma} = \frac{9z_3+36 z_1 z_2+16 z_1^3}{Q_2^2} D_x.
 \end{align*}

 \begin{theorem}\label{th:invarEq29}
Let $\E\subset J^k$ be a $\mathfrak g$-invariant determined ODE system as in \eqref{eq:GeneralSystem},
given by functions $F \in \mathcal{O}(J^k),G \in \mathcal{O}(J^l)$ of orders $k$ and $l\leq k$. Then,
either $F$ and $G$ can be expressed through rational absolute differential invariants or
the system takes one of the following forms:
\begin{itemize}
 \item $G=R_2$ and $F=Q_2$ or $F$ is a function of the conditional absolute invariants which are generated by $J_3$ and  $\nabla_\Sigma$.
\end{itemize}
 \end{theorem}

\section{Miscellaneous results}\label{Sect:Misc}

The discussion and results of this section are closely related to the results of \S\ref{sect:Scalar} and \S\ref{sect:Sytems}. In \S\ref{sl2} we discuss some subtleties related to our computations that one should be aware of when interpreting the results. In \S\ref{sect:twistor} we point out a correspondence between the two $\mathfrak{sp}(4)$ realizations on $\mathbb C^3$ that can be understood either through twistor correspondence or jet-prolongation. \S\ref{sect:nonfree} focuses on a particular underdetermined PDE from \S\ref{sect:sp4} where the $\mathfrak{sp}(4)$ action never becomes free. Lastly, in \S\ref{sect:moduli} we sketch an argument explaining that the space of ODE systems with essentially point (or essentially contact) symmetry algebras is small compared to the space of all ODE systems admitting infinitesimal symmetries.

\subsection{Variation on $\mathfrak{sl}(2)$ realizations}\label{sl2}
In this section we show by examples how different realizations of a given abstract Lie algebra can have significantly different invariant differential equations. Consider the following five realizations of $\mathfrak{sl}(2)$:
\begin{align*}
 \mathfrak g_1 &= \langle y \partial_x, x \partial_x - y\partial_y, x \partial_y \rangle,\\
 \mathfrak g_2 &= \langle \partial_x, x \partial_x-y\partial_y, x^2 \partial_x-2xy\partial_y \rangle, \\
 \mathfrak g_3 &= \langle \partial_x, x \partial_x-\partial_y,x^2\partial_x-2x\partial_y \rangle, \\
 \mathfrak g_4 &= \langle \partial_x, x\partial_x-y\partial_y, x^2 \partial_x-(2xy+1)\partial_y \rangle, \\
 \mathfrak g_5 &= \langle \partial_x + \partial_y, x\partial_x+y\partial_y,x^2\partial_x+y^2\partial_y, \rangle .
\end{align*}

The first three Lie algebras of vector fields are locally equivalent. 
The Lie algebra $\g_1$ is the standard linear $\mathfrak{sl}(2)$-representation.
It is transitive on the complement of the 0-dimensional orbit $\{(0,0)\}$ in $J^0$
and preserves the one 1-dimensional distribution $\langle x\p_x+y\p_y\rangle$.
The Lie algebra $\g_2$ has one 1-dimensional orbit, given by $y=0$, and one 2-dimensional orbit on $J^0$.
The Lie algebra $\g_3$ is transitive. Both $\g_2$ and $\g_3$ preserve one 1-dimensional distribution, given by $\langle\p_y\rangle$.

The Lie algebras $\g_4$ and $\g_5$ are also locally equivalent to each other, and not locally equivalent to the previous Lie algebras.
The first one is transitive, and the second has a singular orbit $y=x$ and is transitive in the complement.
They have two invariant 1-dimensional distributions: $\g_4$ preserves the line distributions $\langle\p_y\rangle$, $\langle\p_x+y^2\p_y\rangle$,
while $\g_5$ preserves $\langle\p_y\rangle$, $\langle\p_x\rangle$.

Table \ref{tab:sl2} gives generators for the field of rational differential invariants for each of these five Lie algebras, expressed in terms of relative differential invariants. It also shows the Lie determinant, up to a constant factor. It clearly illustrates some subtleties regarding the existence of first-order relative invariants, and the algebraic type of the absolute invariants.

\begin{quotation}
{\bf Observation:} Several properties, such as the number of first-order relative differential invariants, the algebraic type (rational/polynomial) of absolute differential invariants, and the order of the basic invariant derivations, are sensitive to the choice of global realization as a Lie algebra of vector fields, meaning that they may differ for locally equivalent realizations. Furthermore, the number of first-order invariant ODEs is reliant on the way we split the coordinates in $J^0$ into ``dependent'' and ``independent'' ones.
\end{quotation}

\begin{table}[h]
\begin{center}
\begin{tabular}{|l|l|l|l|}
\hline
\rule[-1ex]{0pt}{2.5ex} & Relative invariants & Lie determinant & Absolute invariants \\
\hline
\rule[-3ex]{0pt}{6.5ex} $\mathfrak g_1$ & \begin{tabular}{l} $R_1= x y_1-y$ \\ $R_2=y_2$ \end{tabular} & $R_1^2$ &  $\mathcal A_1 = \left\langle\frac{R_2}{R_1^3}, \frac{1}{R_1} D_x \right\rangle$ \\
\hline
\rule[-3ex]{0pt}{7.5ex}  $\mathfrak g_2$ & \begin{tabular}{l} $R_0= y$ \\ $R_2=2yy_2-3y_1^2$ \end{tabular}&  $R_0^2$ & $\mathcal A_2 = \left\langle\frac{R_2}{R_0^4}, \frac{1}{R_0} D_x \right\rangle$ \\
\hline
\rule[-1ex]{0pt}{4.5ex}  $\mathfrak g_3$ & & $ 1$ & $\mathcal A_3 = \left\langle(2 y_2-y_1^2) e^{-2y},  e^{-y} D_x\right\rangle$ \\
\hline
\rule[-3ex]{0pt}{7.5ex}  $\mathfrak g_4$ & \begin{tabular}{l} $R_1= y_1 - y^2$ \\ $R_2=y_2-6y y_1+4y^3$ \end{tabular} & $R_1$ & $\mathcal A_4 = \left\langle\frac{ R_2^2}{R_1^3}, \frac{ R_2}{R_1^2} D_x\right\rangle$  \\
\hline
\rule[-3ex]{0pt}{7.5ex}  $\mathfrak g_5$ & \begin{tabular}{l} $R_0=y-x$ \qquad \qquad  $R_1= y_1$ \\ $R_2=(y-x)y_2-2y_1(y_1+1)$ \end{tabular} & $R_0^2 R_1$ &  $\mathcal A_5 = \left\langle\frac{R_2^2}{R_1^3}, \frac{R_0 R_2}{R_1^2} D_x\right\rangle$  \\
\hline
\end{tabular}
\end{center}

\caption{\vphantom{$\dfrac{A}{A}$}%
Examples showing how the absolute and relative differential invariants can depend on the specific realization of a Lie algebra.}
\label{tab:sl2}
\end{table}


The first-order relative invariant for Lie algebra $\g_1$ corresponds to the invariant 1-dimensional distribution:
the ODE $xy_1=y$ has solutions $y=C x$, which are integral manifolds away from the singularity $(0,0)$.
For the Lie algebras $\g_2$ and $\mathfrak g_3$, the integral manifolds of the invariant distributions are given by $x=C$.
These are not described by a function $y=y(x)$, hence they are not solutions of first-order ODEs in normal form.
This is an example of how a first-order relative invariant can be removed by a point transformation, as we discussed at the end of \S\ref{sect:Lie16}.
The singular $\g_2$-orbit $y=0$ is given by the vanishing of the Lie determinant, while in the case of $\g_3$ this is moved to infinity.
The singular $\g_1$-orbit has codimension 2, and is therefore not given by a scalar condition. Actually, it is the singularity of the invariant distribution.

For Lie algebras $\g_4$, $\g_5$ we have already rectified one invariant distribution to be vertical, so it is invisible through a relative invariant.
The other distribution is specified by vanishing of the Lie determinant: $y_1=y^2$ or $y_1=0$ respectively.
The singular orbit is moved to infinity for $\g_4$ and is given by an order 0 factor of the Lie determinant in the case of $\g_5$.

\subsection{Prolongation-projection via twistor correspondence} \label{sect:twistor}

The primitive $\mathfrak{sl}(2)\ltimes\C^3$-reali\-zation preserving the Minkowski metric, which we studied in \S\ref{sect:Minkowski},
can be extended to an $\mathfrak{sp}(4)$-realization by adding the remaining symmetries of the invariant differential equation $y_1^2+z_1^2-1=0$.
In addition to \eqref{L3}, the four remaining generators of this $\mathfrak{sp}(4)$-realization are
 \begin{gather}
x \partial_x+y\partial_y+z\partial_z, \qquad (x^2+y^2+z^2)\partial_x+2xy \partial_y+2xz\partial_z, \label{L3+}\\
2xy \partial_x+(x^2+y^2-z^2)\partial_y+2yz \partial_z, \qquad  2xz\partial_x+2yz\partial_y+(x^2-y^2+z^2) \partial_z.\label{L3++}
 \end{gather}
This $\mathfrak{sp}(4)$-realization is different from the one studied in \S\ref{sect:sp4}.
Indeed, this one is irreducible, while that of \eqref{L5} preserves the contact structure $dx-z dy+y dz=0$ (see Figure \ref{fig:Primitive}).

\begin{prop}
Two non-equivalent realizations of $C_2=\mathfrak{sp}(4)$ in $\mathbb C^3$, corresponding to the flag manifolds $C_2/P_1$ and $C_2/P_2$, are equivalent after a lift of the action to $\mathbb C^4$. This lift can be interpreted both as twistor correspondence and jet-prolongation.
\end{prop}

Twistor correspondences generalizing Penrose's ideas were introduced to parabolic geometries by A. \v{C}ap in \cite{Cap}, but the relation to jet formalism discussed here is novel.

\begin{proof}
Changing the contact structure to its Darboux normal form $dy-y_1dx$, in a new space $J^1(\C)=\C^3(x,y,y_1)$, results
in an equivalent realization $\g$ of \eqref{L5}. Prolonging this Lie algebra of vector fields to $J^2(\C)=\C^4(x,y,y_1,y_2)$
we get the contact symmetry algebra of the trivial ODE $\E=\{y_3=0\}\subset J^3$. In fact, this prolonged Lie algebra $\hat\g$ preserves the
following splitting of the rank 2 distribution on $\E\simeq J^2$: $\Delta=\langle\D_x^\E=\p_x+y_1\p_y+y_2\p_{y_1}\rangle\oplus\langle\p_{y_2}\rangle$.

Projection along $\p_{y_2}$ sends us back to $J^1$, while the projection along $\D_x^\E$ transforms this $\hat\g$ into a primitive Lie algebra of
vector fields in 3D. Indeed, the first integrals are $I_0=y_0-xy_1+\tfrac12x^2y_2$, $I_1=y_1-xy_2$, $I_2=y_2$, so passing from
$(x,y,y_1,y_2)$ to new coordinates $(x,y,z,q):=(I_0,I_1,I_2,x)$ we get a Lie algebra projectable along $\p_q$. The resulting Lie algebra $\tilde\g$
on $M^3=\C^3(x,y,z)$ is the Lie algebra of conformal Killing vectors for the Lorentzian metric $g=dy^2-2dxdz$, and changing coordinates to have the
normal form $dx^2-dy^2-dz^2$ we get the equivalent Lie algebra of vector fields given by \eqref{L3}-\eqref{L3+}-\eqref{L3++}.
Conversely $\tilde\g$ lifts to the bundle of scales for the conformal structure $[g]$ to give $\hat\g$.

It turns out that we can identify (after completion) $J^1$ with $C_2/P_1$, $J^2$ with $C_2/P_{1,2}$ and $M^3$ with $C_2/P_2$,
where $C_2=Sp(4)$ and $P_\sigma$ are parabolic subgroups marked by crosses on the Dynkin diagrams below.
The three realizations of $C_2$ on generalized flag manifolds $C_2/P_\sigma$ as discussed above,
are conveniently related by the following double fibration:
 \begin{figure}[h]
 \begin{center}
 \begin{tikzpicture}
 \draw (0,0.05) -- (1,0.05);
 \draw (0,-0.05) -- (1,-0.05);
 \draw (0.6,0.15) -- (0.4,0) -- (0.6,-0.15);
 \node[draw,circle,inner sep=2pt,fill=white] at (0,0) {};
 \node[draw,circle,inner sep=2pt,fill=white] at (1,0) {};
 \node[below] at (0,-0.1) {$\times$};
 \draw (3,0.05) -- (4,0.05);
 \draw (3,-0.05) -- (4,-0.05);
 \draw (3.6,0.15) -- (3.4,0) -- (3.6,-0.15);
 \node[draw,circle,inner sep=2pt,fill=white] at (3,0) {};
 \node[draw,circle,inner sep=2pt,fill=white] at (4,0) {};
 \node[below] at (4,-0.1) {$\times$};
 \draw (1.5,1.05) -- (2.5,1.05);
 \draw (1.5,0.95) -- (2.5,0.95);
 \draw (2.1,1.15) -- (1.9,1) -- (2.1,0.85);
 \node[draw,circle,inner sep=2pt,fill=white] at (1.5,1) {};
 \node[draw,circle,inner sep=2pt,fill=white] at (2.5,1) {};
 \node[below] at (1.5,0.9) {$\times$};
 \node[below] at (2.5,0.9) {$\times$};
\path[->,>=angle 90](1.9,0.75) edge (0.65,0.25);
\path[->,>=angle 90](2.1,0.75) edge (3.35,0.25);
 \end{tikzpicture}
\end{center}
\end{figure}
The arrows are projections corresponding to the inclusions $P_1\hookleftarrow P_{12}\hookrightarrow P_2$.
This twistor correspondence represents the above jet-picture with prolongations and lifts.
\end{proof}

\subsection{An underdetermined ODE with 
non-free $\mathfrak{sp}(4)$-symmetry} \label{sect:nonfree}

In \S\ref{sect:sp4} we saw that the Lie algebra of vector fields \eqref{L5} does not become free on the underdetermined ODE given by $R_3= z_2 y_3 - y_2 z_3=0$. More precisely, we have the following statement.
\begin{theorem}
Consider the action of the Lie algebra $\g=\mathfrak{sp}(4)$ of vector fields \eqref{L5}  prolonged to $J^k(\mathbb C,\mathbb C^2)$. All orbits contained in the invariant subset
\[\Pi^k = \{R_3 = 0, D_x(R_3)=0, \dots , D_x^{k-3}(R_3)=0\} \subset J^k(\mathbb C,\mathbb C^2)\]
have dimension less than or equal to 9 for every $k \geq 3$.
\end{theorem}
\begin{proof}
The stabilizer of a point $\theta \in \Pi^k$ for $k \geq 4$ is 1-dimensional. This 1-dimensional Lie subalgebra depends on $\theta_2=\pi_{k,2}(\theta)$ and is spanned by
 \[
U=\alpha (\alpha X_1+y_2 X_2+z_2 X_3+\beta X_7)+z_2^2 X_4-y_2^2 X_5+y_2 z_2 X_6+\beta (z_2 X_8-y_2 X_9+\beta X_{10}),
 \]
where $\alpha$ and $\beta$ are given by
 \[
\alpha = x (y_1 z_2- z_1 y_2)+z y_2 -y z_2, \qquad \beta = z_1 y_2-y_1 z_2,
 \]
and $X_1,\dots, X_{10}$ are the basis elements as expressed in \eqref{L5}, in the order they appear there.

Let us give a non-computational explanation of this fact. The equation $R_3=0$ describes plane curves
$\gamma\subset\mathbb{P}^2\subset\mathbb{P}^3$ (above we used an affine chart
$\C^3(x,y,z)\subset\mathbb{P}^3$, with a plane $\mathbb{P}^2=\{ax+by+cz=d\}$
and $y=y(x)$, $z=z(x)$),
where $\mathbb{P}^3=Sp(4)/P_1$ is the homogeneous representation with the stabilizer of a plane
being the first parabolic subgroup. The action of $Sp(4)$ is the projectivization of the standard linear action
on the symplectic $\C^4$, and $P_1$ is the stabilizer of a line $\ell$ (for which the plane $\mathbb{P}^2$ is
the projectivization of the skew-orthogonal complement $\ell^\perp$).
Note that this gives $\mathfrak{p}_1=\op{Lie}(P_1)$ a filtration, with the corresponding grading of
$\g=\mathfrak{sp}(4)$ as follows:
 $$
\g=\g_{-2}\oplus\g_{-1}\oplus\underbrace{\g_0\oplus\g_1\oplus\g_2}_{\mathfrak{p}_1}.
 $$
The reductive part $\g_0=\mathfrak{gl}(2)$ acts on $\ell^\perp/\ell\simeq\C^2$ in the standard way,
and the nilradical $\g_1\oplus\g_2=\mathfrak{heis}(3)=\C^2\ltimes\C$ maps $\ell^\perp\to\ell$
so that $\g_1=(\ell^\perp/\ell)^*\simeq\C^2$ and $\g_2$ acts trivially on $\ell^\perp$.
This 1-dimensional space is generated by the vector $U$ above, which depends only on $\theta_2\equiv\ell$.
\end{proof}

Since the set $\Pi^k$ is meager in $J^k(\mathbb C,\mathbb C^2)$, these computations are consistent with Theorem 7.1 in \cite[Th. 7.1]{AO} which states that the prolonged action becomes free on a comeager subset of $J^k(\mathbb C,\mathbb C^2)$ for sufficiently large $k$. However, this example shows that the computation of conditional  absolute invariants on $\Pi^k$ will involve a Lie algebra of vector fields that does not become free after prolongation. One  consequence of this is that the method of moving frames (see \cite{FO}) can not be applied directly for finding conditional differential invariants in this case.

Notice that the dimension of $\Pi^k$ grows without bound with $k$, making this example significantly different from Example 3 of \cite{AO}, where the lack of freeness is a simple observation. More generally, the action of the Lie algebra $\mathfrak{so}(n)\ltimes\C^n$ of the motion group, or a larger Lie algebra like $\mathfrak{aff}(n)$ or $\mathfrak{sl}(n+1)$, on the space
of (unparametrized) straight lines in $\C^n$, given by the equation $\{y^j_2=0:1\leq j\leq n-1\}\subset J^2(\C,\C^{n-1})$
is non-free for $n\geq 2$ by dimensional reasons. These reasons are absent in the above example, where the prolonged
underdetermined equation $\Pi^\infty$ is infinite-dimensional.

We can also get a determined equation
by imposing some $\mathfrak{sp}(4)$ invariant constraint from Theorem \ref{tm24}: Take for example the determined ODE system defined by $R_3=0$ and $\nabla_\Pi^i(K_6)=\Psi(K_5,K_6,\nabla_\Pi(K_6),\dots,\nabla_\Pi^{i-1}(K_6))$. To conclude:
\begin{cor}
There exist $\g$-invariant (determined) ODE systems (depending on numeric and functional parameters) with arbitrary large but finite dimension, on which the action is not free.
\end{cor}

We point out that this phenomenon is observed multiple times in our computations. For the Lie algebra action \eqref{L6}, which we treated in \S\ref{sect:Lie16}, the orbits on the underdetermined ODE $\Pi^k \subset J^k(\mathbb C,\mathbb C^2)$ have dimension less than or equal to 5 for every $k$ even though the Lie algebra is 6-dimensional. For the 8-dimensional Lie algebra \eqref{L8}, treated in \S\ref{sect:Lie29}, the orbits on $\Sigma^k \subset J^k(\mathbb C, \mathbb C^2)$ have dimension $\leq 6$ for every $k$.

\subsection{Moduli of ODEs with infinitesimal symmetries}\label{sect:moduli}

This section is somewhat speculative, so we will be slightly vague and only sketch the ideas here.
They can be rigorously justified, but we omit it for the sake of brevity.
An illustrating example is given at the end of the section.

Consider the set of all $k$th-order (determined) algebraic ODE systems in $m$ independent variables.
This means they can be defined by $m$ independent equations $F_1=0, \dots, F_m=0$, where $F_i$ are polynomials of some degree $\leq p$ on fibers $J^k\to J^j$.
Here $j=0$ or $j=1$ depending on whether one considers point or contact symmetries. This space of ODE systems
is parametrized by a fixed number of functions $f_i$ in  $\mathcal{O}(J^0)$ or $\mathcal{O}(J^1)$, respectively.
(The number of functions can be computed from the numbers $m$, $k$ and $p$.)

We consider the systems up to equivalence under the action of a pseudogroup $\mathcal{G}$, where $\mathcal{G}$ is either the contact pseudogroup
$\op{Cont}_{\text{loc}}(J^1)$ in the scalar ($m=1$) higher order ($n>2$) case or the point pseudogroup $\op{Diff}_{\text{loc}}(J^0)$ in general.
We fix $m$ in what follows.

For fixed $k,p$ let $\mathcal{M}^k_p$ denote the subset of all such ODE systems having at least one non-trivial infinitesimal symmetry.
Every $\mathcal{M}^k_p$ can be considered as the solution space of a PDE system on functions $f_i$. Let $\mathcal{M}=\cup_{k,p}\mathcal{M}^k_p$.
The moduli space $\mathcal{M}/\mathcal{G}$ of invariant ODE systems can be understood as the solution space of
the quotient equation of this (possibly disconnected) PDE system.
The set $\mathcal{M}/\mathcal{G}$ has complicated non-Hausdorff topology. Nevertheless, consider now the subset $\mathcal{M}'\subset\mathcal{M}$
of ODE systems with essentially contact or essentially point symmetries and the corresponding quotient $\mathcal{M}'/\mathcal{G}$.

 \begin{theorem}
The subset $\mathcal{M}'\subset\mathcal{M}$ is meager, and moreover $\mathcal{M}'/\mathcal{G}\subset\mathcal{M}/\mathcal{G}$ is non-separable.
 \end{theorem}

 \begin{proof}
Fix numbers $k,p$.
If the symmetry algebra $\g$ of an ODE system is contact irreducible or primitive, consider a subalgebra $\mathfrak{h}$ that does
not possess these properties. Such $\mathfrak{h}$ always exists, for instance a maximal Abelian subalgebra will work.
The set of ODE systems invariant with respect to $\mathfrak{h}$ is much larger: $i$-jets of those form an algebraic variety with
$\g$-invariant equations being a proper Zariski closed subset for large $i$.
Also, the $\mathcal{G}$-action on a representative ODE system with $\g$-symmetry yields a smaller and lower-dimensional
set compared to that with $\mathfrak{h}$-symmetry: the orbit is larger if the isotropy is smaller.
Now uniting these singular orbits over $k,p$ proves both statements.
 \end{proof}

One can also restrict to subset $\mathcal{M}_{(l)}$ of ODEs with symmetry algebra of dimension at least $l$,
and the above statement holds true  when replacing $\mathcal{M}$ with $\mathcal{M}_{(l)}$. Let us briefly indicate the idea for the case $m=1$ and $m=2$.
The proof holds literally for $l=1,2,3,4$ as a non-irreducible, non-primitive subalgebra of such dimension exists.
For $l$ in the intermediate range 5 to 15 the claim is based on the following observation:
a common subalgebra (for instance $\langle\p_x,\p_y\rangle$) allows both irreducible/primitive as well as
the opposite extensions. The latter however are less rigid and hence carry more parameters (here we do not
restrict order of the equation, or restrict it by a sufficiently large number).
For $l>15$ the claim is trivially true because there are no such irreducible/primitive Lie algebras of
vector fields in 2- or 3-spaces.

\textbf{Example:} Consider the space of ODEs defined by equations of the form
\[ y_2 =f(x,y). \]
The vector field $X=a(x,y) \partial_x+b(x,y) \partial_y$ is a point symmetry if and only if
\[X^{(2)}\left(y_2-f(x,y)\right)\Big|_{y_2=f(x,y)} = 0.\] The left-hand-side is a cubic polynomial in $y_1$ whose vanishing is equivalent to the vanishing of its four coefficients. The system obtained in this way can be written
\begin{equation} \label{eq:sym}
\begin{split}
 a_{xx} = 2 b_{xy}-3  f a_{y} , \qquad
 a_{yy} &= 0, \\
 b_{xx} = f_x a+ 2f a_x+ f_y b- f b_{y}, \qquad
 b_{yy} &= 2 a_{xy}.
 \end{split}
\end{equation}

Since system \eqref{eq:sym} is overdetermined, there exist compatibility conditions.
If all of them are satisfied (for some class of functions $f$), we get a finite type system with 8-dimensional solution space. When the compatibility conditions are not satisfied, we get more constraints on $(a,b)$, and thus the dimension of the solution space shrinks.
Prolonging \eqref{eq:sym} to the 4-jets
and eliminating derivatives of $a,b$ of orders 2, 3 and 4, we get the following two constraints :
\begin{equation} \label{eq:compatibility1}
a_y f_{yy}=0, \qquad 2 f_{yy} a_x+f_{yy} b_y+f_{xyy} a+f_{yyy} b=0.
\end{equation}
These compatibility conditions are satisfied if and only if $f$ is a solution to
\begin{equation}
 f_{yy} = 0.  \label{eq:compsl3}
 \end{equation}
Then by the Lie-Liouville criterion the equation is trivializable, i.e.\ point equivalent to $y_2=0$.

If we assume $f_{yy}\not \equiv 0$, then $a_y=0$ and thus $a_{xy}=0$.  The equation $a_y=0$ together with equations 1, 2 and 4 of \eqref{eq:sym} implies that $a=a(x)$, $b=b(x)+(\tfrac12a'(x)+c) y$. Then the second equation of \eqref{eq:compatibility1} reduces to the derivative of the remaining equation of \eqref{eq:sym}. It follows that the ODE has $\leq 6$ independent symmetries and further compatibility analysis reduces this number to at most 2 (it is a known fact that symmetry algebras of second-order scalar ODEs can have dimension 0, 1, 2, 3 or 8).

Alternatively, compute the $k$-th prolongation of \eqref{eq:sym} and eliminate $a_{xx},a_{yy},b_{xx},b_{yy}$
and higher-order derivatives to obtain a linear system of equations on $a,b,a_x,a_y,b_x,b_y,a_{xy},b_{xy}$
with coefficients depending on the $(k+1)$-jet of $f$. The rank of this sequence of systems will stabilize at some integer $r \leq 8$. If $y_2=f(x,y)$ has symmetries, then $r \leq 7$. Imposing this condition on $r$ puts constraints on $f$, which for our particular class of ODEs defines the set $\mathcal{M}$ of symmetric ODEs.
The first of such constraints is (where we denote $f_{nm}=\p_x^n\p_y^mf$)
 \begin{multline*}
(3f_{02}f_{13}-2f_{12}f_{03})f_{05}^2 +(f_{12}f_{04}^2 -(3f_{02}f_{14} +f_{03}f_{13})f_{04} -f_{02}f_{03}f_{15}
+2f_{03}^2f_{14})f_{05} \\
+2f_{02}f_{04}^2f_{15}
 - (2f_{02}f_{13}f_{06} -f_{12}f_{03}f_{06} +f_{03}^2f_{15})f_{04} + f_{02}f_{03}f_{14}f_{06}=0.
 \end{multline*}

\section{Conclusion}\label{OutL}

In this paper we demonstrated that scalar ODEs and ODE systems (in two unknowns) with the symmetry
algebra being essentially contact or essentially point are special among all symmetric differential equations,
and we listed them all via (absolute, relative or conditional) differential invariants.
This follows the general approach of Sophus Lie \cite{L1,L2}, and for scalar ODEs was already
considered in \cite{WMQ}. Our approach is more of a global nature. In particular we distinguish between
algebraic and analytic ODEs, though the generating invariants can always be chosen rational in higher jets.
For systems (pairs of ODEs) our results are entirely new.

Thus the main bulk of ODEs with infinitesimal symmetries constitute equations having
essentially fiber-preserving symmetries, in particular their equivalence problem can be studied with
respect to fibre-preserving transformations of variables. This indeed already attracted an interest
for scalar ODEs in lower orders, see for instance \cite{HK,G,GN}.

A generalization of our results to ODE systems with more dependent variables is possible:
one needs a classification of primitive Lie algebras of vector fields, and this is indeed available.
In \cite{P} J.\ Page continued Lie's program and classified primitive Lie algebras of vector fields in four dimensions. 
There is also a classification of primitive pairs $(\g,\mathfrak{h})$ of Lie algebras, such that $\mathfrak{h}$
is a maximal subalgebra of $\g$, due to Morozov and Dynkin \cite{M,D}. This leads to a
classification of primitive actions on spaces of any dimension, see \cite{Go,Dr}.

The procedure in \S\ref{sect:algorithm} can be applied for differential equations with more dependent and/or independent variables with some modifications (it can work for PDEs as well).
We already noticed an increase in complexity when passing from scalar ODEs to pairs of ODEs.
More complexity issues come in higher dimensions: one can encounter vector relative invariants with
values in non-trivial bundles and conditional-conditional invariants, etc.

One can also approach the classification in the real case.
Already Sophus Lie started this program, and there exist several versions of his real classification,
see e.g.\ \cite{GKO}. The list of primitive Lie algebras of vector fields in real plane is indeed longer
than that in complex plane.
The real classification of irreducible Lie algebras of contact vector fields was attained too, together with computation of
fundamental differential invariants, see \cite{DK}. However the complete classification in the real
case is considerably more complicated, as one also encounters continuous parameters and more branching.
Moreover a possible generalization to real smooth case can be only attained locally near generic/regular points,
so some singular equations will not be covered.

\appendix

\section{Finite-dimensionality of symmetry algebra}\label{ApA}

In this section we work over $\R$. Indeed, working over $\C$ in the main bulk of this paper
was important for classification issues, but otherwise one can consider the real case.
The purpose of this appendix is to give a criterion for finite-dimensionality of the symmetry algebra
of an ODE system (we restrict to the scalar case and pairs of equations). This justifies our
application of the classification of finite-dimensional Lie algebras of vector fields in 2 and 3 dimensions. We prove the following propositions\footnote{Proposition \ref{prop:finitedim1} is well-known, see for example \cite{Sok}. Finite-dimensionality of the symmetry algebra of systems of ODEs of the same order is proven for example in \cite{DKM},  see also the discussion and references in \cite[p.206]{O}. We are unaware of a similar proof for systems of ODEs of different orders; in addition our proof differs from loc.cit.}:
\begin{prop} \label{prop:finitedim1}
(1) The Lie algebra of point symmetries of a scalar ODE of order $\geq 2$ is finite-dimensional. (2) The Lie algebra of contact symmetries of a scalar ODE of order $>2$ is finite-dimensional.
\end{prop}

\begin{prop} \label{prop:finitedim2}
The Lie algebra of (point) symmetries of a system of 2 ODEs (with 2 independent variables) of orders $(k,l)$, $k \geq l \geq 1$, is finite-dimensional provided that either $l>1$ holds or that $l=1$, $k>1$ and condition \eqref{eq1d} holds for the differential equation of order 1.
\end{prop}

The conditions of Propositions \ref{prop:finitedim1} and \ref{prop:finitedim2} are satisfied for all equations arising in our paper, except for the systems given by 2 equations of first order. Specifically, let us mention that in the case of ODE systems, the case $l=1$ is realized for special cases given by the relative invariant $R_1=0$, and the first part of condition \eqref{eq1d} holds for $R_1$ of \S\ref{sect:Minkowski} and \S\ref{sect:so4}, while the second part of this condition holds for $R_1$ of \S\ref{sect:sp4}, \S\ref{sect:Lie16} and \S\ref{sect:Lie27}. On the other hand, the system $R_1=0,Q_1=0$ from \S\ref{sect:Lie27} and the system $y_1=0,z_1=0$ appearing in Remark \ref{rk:point} has infinite-dimensional symmetry algebra.

\begin{proof}
Let us start with scalar ODEs of principal type, i.e.\ of the form $y_n=f(x,y,y_1,\dots,y_{n-1})$,
and assume $n>1$ (all scalar ODEs of order $n=1$ are locally equivalent and have infinite-dimensional
symmetry algebra).
Geometrically this equation $\E$ can be identified with the space $J^{n-1}(\R)\simeq\R^{n+1}$
equipped with rank 2 distribution $\Delta$ spanned by
 $$
\D^\E_x=\p_x+y_1\p_y+y_2\p_{y_1}+\dots+y_{n-1}\p_{y_{n-2}}+f\p_{y_{n-1}}\,\text{ and }\ \p_{y_{n-1}}.
 $$
The splitting $\Delta=\langle\D^\E_x\rangle\oplus\langle\p_{y_{n-1}}\rangle$
encodes consideration of $(\E,\Delta)$ modulo point transformations for $n=2$ and contact
transformations for $n\ge3$. The weak derived flag of this {\it nonholonomic} distribution
$\Delta_1=\Delta$, $\Delta_i=[\Delta,\Delta_{i-1}]$ is filtered and the associated graded space
$\mathfrak{m}=\sum\g_{-i}$, $\g_{-i}=\Delta_i/\Delta_{i-1}$, is naturally a graded nilpotent Lie algebra,
called the {\it symbol} of the distribution or its {\it Carnot algebra}, cf.\ \cite{AK}.
For the above $\Delta$ it is
 $$
\mathfrak{m}=\g_{-n}\oplus\dots\oplus\g_{-1}
\ \text{ with }\ \dim\g_{-n}=\dots=\dim\g_{-2}=1, \ \dim\g_{-1}=2.
 $$
The Lie algebra structure on $\mathfrak{m}$ is unique of so-called Goursat type. The split condition
is equivalent to reduction of $\mathfrak{der}_0(\mathfrak{m})$, which is $\mathfrak{gl}(2,\R)$ for $n=2$ and
a Borel (3D) subalgebra $\mathfrak{b}\subset\mathfrak{gl}(2,\R)$ for $n>2$,
to $\g_0=\R\oplus\R$ (rescalings of the two directions).

The maximal graded Lie algebra $\g$ with $\g_{\leq0}=\mathfrak{m}\oplus\g_0$,
such that $\g_{\ge0}$ acts effectively on $\mathfrak{m}$, is called the
{\it Tanaka prolongation} $\g=\op{pr}(\mathfrak{m},\g_0)$, see \cite{T}.
(In the case $\mathfrak{m}$ is fundamental, i.e.\ generated by $\g_{-1}$, the effectiveness
is equivalent to $[v,\g_{-1}]=0$, $v\in \g_k$, $k\ge0$ $\Rightarrow$ $v=0$.)

The Tanaka prolongation of the above algebra is finite-dimensional by a general criterion from \cite{T}.
Actually it is equal to $\g=\g_{-2}\oplus\dots\oplus\g_2\simeq\mathfrak{sl}(3,\R)$ for $n=2$,
$\g=\g_{-3}\oplus\dots\oplus\g_3\simeq\mathfrak{sp}(4,\R)$ for $n=3$ and
$\g=\g_{-n}\oplus\dots\oplus\g_1\simeq\mathfrak{gl}(2,\R)\ltimes\R^n$ for $n\ge4$.
By \cite{T} the maximal symmetry dimension of $(\E,\Delta)$ (with splitting reduction of the structure group)
is bounded by $\dim\g$. Therefore the symmetry dimension of $n$-th order scalar ODE is bounded by 8 for $n=2$,
by 10 for $n=3$ and by $n+4$ for $n\ge4$.
Moreover the existence of the grading element implies uniqueness of the maximal symmetry model.

Similarly for systems of pairs of ODEs of orders $k\ge l>1$
 $$
y_k=f(x,y,\dots,y_{k-1},z,\dots,z_{l-1}),\
z_l=h(x,y,\dots,y_l,z,\dots,z_{l-1})
 $$
considered as submanifolds $\E\subset J^k$ (the second equation being prolonged $k-l$ times)
the geometry with respect to point transformations is encoded\footnote{Using embedding in the mixed jet-space,
the encoded geometry is different \cite{AK,DZ}: the Carnot algebra $\mathfrak{m}$ associated to the
split distribution $\Delta=\langle\D^\E_x\rangle\oplus\langle\p_{y_{k-1}},\p_{z_{l-1}}\rangle$ and its weak derived flag
has graded components with $\dim\g_{-k}=\dots=\dim\g_{-l-1}=1$, $\dim\g_{-l}=\dots=\dim\g_{-2}=2$, $\dim\g_{-1}=3$.}
via the filtration of $T\E=\Delta_k$: $\Delta_{k-i}=d\pi_{k,i-1}^{-1}(\langle d\pi_{k,i-1}\D^\E_x\rangle)$, $i=1,\dots,k-1$,
where $\D^\E_x$ is the total derivative on the equation. This distribution is split into horizontal and vertical parts as follows:
 $$
\Delta_{k-i}=\langle\D^\E_x\rangle\oplus(\langle\p_{y_j},\p_{z_j}\rangle_{j=i}^{k-1}\cap T\E).
 $$
The corresponding Carnot algebra (for $k\neq l$ it is not fundamental) equals
 \begin{multline*}
\mathfrak{m}=\g_{-k}\oplus\dots\oplus\g_{-l}\oplus\dots\oplus\g_{-1}\ \text{ with }\\
 \begin{array}{l}
\dim\g_{-k}=\dots=\dim\g_{l-k-1}=2, \ \dim\g_{l-k}=\dots=\dim\g_{-2}=1, \ \dim\g_{-1}=2 \ \text{ for $k>l$};\\
\dim\g_{-k}=\dots=\dim\g_{-2}=2, \ \dim\g_{-1}=3 \ \text{ for $k=l$}.
 \end{array}
 \end{multline*}
Reduction of the structure group corresponds to $\g_0=\R\oplus\mathfrak{gl}(2,\R)$ for $k=l$ and
$\g_0=\R\oplus\mathfrak{b}\subset\R\oplus\mathfrak{gl}(2,\R)$ for $k>l$.
The Tanaka prolongation $\g=\op{pr}(\mathfrak{m},\g_0)$ is again finite-dimensional for $l>1$,
moreover it can be explicitly computed and the symmetry dimension is as follows, cf.\ \cite{DZ,KT}:
 $$
\dim\g=\left\{
\begin{array}{ll}
9+\binom{k+1}{2}+3\delta_{k,l}& \text{ for }k\geq l=2;\\
5+l+\left[\frac{k+l-2}{l-1}\right]\left(k-\frac{l-1}2\left[\frac{k-1}{l-1}\right]\right)+\delta_{k,l}
& \text{ for }k\geq l>2.
\end{array}
\right.
 $$
This gives an effective bound on the symmetry dimension of systems of pairs of ODEs.

When $l=1$ the situation is different. If $k=1$ the equation has infinite-dimensional symmetry algebra.
So consider the case $k>1$ with $\E$ given by $y_k=f(x,y,y_1,\dots,y_{k-1},z)$, $z_1=h(x,y,y_1,z)$.
In this case the distribution on the equation has rank 2
(as in the scalar case), namely
 $$
\Delta=\langle\D^\E_x\rangle\oplus\langle\p_{y_{k-1}}\rangle,\ \text{ where }\
\D^\E_x=\p_x+y_1\p_y+\dots+f\p_{y_{k-1}}+h\p_z.
 $$
The distribution is completely nonholonomic if either
 \begin{equation}\label{eq1d}
h_{y_1y_1}\neq0\ \text{ or }\ J:=h_{xy_1}+h\,h_{y_1z}+y_1h_{y\,y_1}-h_y-h_zh_{y_1}\neq0.
 \end{equation}
The symbol of the distribution is
 \begin{equation}\label{eq1m}
\mathfrak{m}=\g_{-k-1}\oplus\dots\oplus\g_{-1}
\ \text{ with }\ \dim\g_{-k-1}=\dots=\dim\g_{-2}=1, \ \dim\g_{-1}=2.
 \end{equation}
However depending on which condition holds in \eqref{eq1d} the Carnot algebra structure changes.
In the first case $h_{y_1y_1}\neq0$ the distribution is of finite type, $\g_0=\R\oplus\R$ and the next prolongation
vanishes. Thus the symmetry dimension is bounded by $\dim\g=k+4$.

In the case $h_{y_1y_1}=0$ but $J\neq0$ we write $h=h_0+h_1y_1$, where $h_i=h_i(x,y,z)$.
Then $J=h_{1x}+h_0h_{1z}-h_{0y}-h_1h_{0z}\neq0$ is the condition that 1-form $\alpha=dz-h_0dx-h_1dy$
in $J^0=\R^3$ is contact. In this case the Carnot structure \eqref{eq1m} is of infinite type, but the structure reduction
$\mathfrak{b}\mapsto\g_0=\R\oplus\R$ makes the prolongation $\g=\op{pr}(\mathfrak{m},\g_0)$ finite-dimensional,
hence the symmetry dimension is bounded by $\dim\g=k+5$.

Finally, if $h_{y_1y_1}=0$ and $J=0$, then the above 1-form $\alpha$ determines an integrable distribution,
which by a point transformation is equivalent to the second equation being $z_1=0$.
In this case the symmetry can be infinite-dimensional, explicitly this happens when $f_z=0$.
\end{proof}

\begin{remark}
From the end of the proof we see that the only candidates for systems with infinite-dimensional symmetry algebra are (equivalent to) the system
 $$
y_k=f(x,y,y_1,\dots,y_{k-1}),\quad z_1=0.
 $$
This system has infinite-dimensional symmetry subalgebra $\{Z(z)\p_z\}$,
and quotient by the corresponding foliation reduces the problem of essential point symmetries to that for scalar ODEs, which we discussed in Section 1.
\end{remark}

\section{Lie algebras of vector fields on $\mathbb C^3$ preserving a 1-dimensional foliation} \label{ApB}



In \cite{L2} Lie listed all Lie algebras of vector fields in $\mathbb C^3$ that preserve a 1-dimensional foliation
\[\varphi(x,y,z)=\text{const}, \qquad  \psi(x,y,z)=\text{const},\]
while not preserving any 2-dimensional foliation of the form \[\Omega(\varphi(x,y,z), \psi(x,y,z)) = \text{const}.\]
The list consists of 21 entries (some of these are families of Lie algebras) that are numbered from 13 to 33. See Satz 3 in Section 44 of \cite{L2} for the main statement. In this subsection, we will refer to these entries by [Lie{\hskip1pt}13] to [Lie{\hskip1pt}33], respectively.  These Lie algebras are the finite-dimensional Lie algebras of vector fields on $\mathbb C^2 \times \mathbb C$ that project to one of the three primitive Lie algebras of vector fields on $\mathbb C^2$.

All the Lie algebras in Lie's list contain one of the following as a Lie subalgebra ([Lie{\hskip1pt}14], \mbox{[Lie{\hskip1pt}13]}, [Lie{\hskip1pt}16])\footnote{[Lie{\hskip1pt}14] is identical with \eqref{L2}.}:
\begin{gather*}
  \langle \partial_x, \partial_y, x \partial_y, y \partial_x, x \partial_x-y\partial_y \rangle, \\
  \langle  \partial_x, \partial_y, x \partial_y+\partial_z, y \partial_x-z^2 \partial_z, x \partial_x-y\partial_y -2z\partial_z \rangle, \\
  \langle \partial_x, \partial_y+x \partial_z, x \partial_y + \tfrac{1}{2} x^2 \partial_z, x \partial_x-y \partial_y, y \partial_x+\tfrac{1}{2} y^2 \partial_z, \partial_z \rangle.
  \end{gather*}
However, the first two of these actually preserves the 2-dimensional foliation $z=\text{const}$, and so do the Lie algebras [Lie{\hskip1pt}15], [Lie{\hskip1pt}19], [Lie{\hskip1pt}20], [Lie{\hskip1pt}21], [Lie{\hskip1pt}22], [Lie{\hskip1pt}26], [Lie{\hskip1pt}28] and [Lie{\hskip1pt}33]. Removing these from our consideration, we are left with 11 remaining families of Lie algebras. Note that the families  [Lie{\hskip1pt}17], [Lie{\hskip1pt}18], [Lie{\hskip1pt}23], [Lie{\hskip1pt}25], [Lie{\hskip1pt}30] and [Lie{\hskip1pt}32]
all contain elements $x^i y^j \partial_z$ for $i+j \leq h$ for some non-negative integer $h$, and they preserve the 2-dimensional foliation $z=\text{const}$ if and only if $h=0$.

\begin{prop}\label{prop:imprimitive}
The finite-dimensional Lie algebras of vector fields in $\mathbb C^3$ preserving a 1-dimen\-sional foliation and no 2-dimension foliation are the following: \textup{[Lie{\hskip1pt}16], [Lie{\hskip1pt}24], [Lie{\hskip1pt}27], [Lie{\hskip1pt}29], [Lie{\hskip1pt}31], {[Lie{\hskip1pt}17]$_{h\geq 1}$}, [Lie{\hskip1pt}18]$_{h\geq 1}$, [Lie{\hskip1pt}23]$_{h\geq 1,a}$, [Lie{\hskip1pt}25]$_{h\geq 1}$,  [Lie{\hskip1pt}30]$_{h\geq 1}$, [Lie{\hskip1pt}32]$_{h\geq 1}$}.
\end{prop}
We have the inclusions $\text{[Lie{\hskip1pt}17]}_{h} \subset \text{[Lie{\hskip1pt}17]}_{h+1}$, $\text{[Lie{\hskip1pt}18]}_{h} \subset \text{[Lie{\hskip1pt}18]}_{h+1}$, $\text{[Lie{\hskip1pt}23]}_{a,h} \subset \text{[Lie{\hskip1pt}23]}_{a,h+1}$, $\text{[Lie{\hskip1pt}25]}_{h} \subset \text{[Lie{\hskip1pt}25]}_{h+1}$,
while similar inclusions do not hold for [Lie{\hskip1pt}30]$_{h+1}$ and [Lie{\hskip1pt}32]$_{h+1}$. However, from the coordinate expressions in \cite{L2} one can easily verify that $\text{[Lie{\hskip1pt}17]}_{h=1}$ is contained in [Lie{\hskip1pt}18]$_h$, [Lie{\hskip1pt}23]$_{a,h}$, [Lie{\hskip1pt}25]$_h$, [Lie{\hskip1pt}30]$_h$ and [Lie{\hskip1pt}32]$_h$ for any $h \geq 1$ and any $a$. Furthermore, it is apparent that [Lie{\hskip1pt}24] contains [Lie{\hskip1pt}16]. This makes it clear that all of the Lie algebras from Proposition \ref{prop:imprimitive} contain as subalgebras either [Lie{\hskip1pt}16], [Lie{\hskip1pt}17]$_{h=1}$, [Lie{\hskip1pt}27], [Lie{\hskip1pt}29] or [Lie{\hskip1pt}31]. These five Lie algebra realizations are given (respectively) by
 \begin{gather*}
\langle \partial_x, \partial_y+x \partial_z, x \partial_y + \tfrac{1}{2} x^2 \partial_z, x \partial_x-y \partial_y, y \partial_x+\tfrac{1}{2} y^2 \partial_z, \partial_z \rangle, \\
\langle \partial_x,\partial_y, x\partial_y, x\partial_x -y\partial_y, y \partial_x, \partial_z, x\partial_z, y\partial_z \rangle, \\
 \langle \partial_x,\partial_y, x \partial_y+\partial_z, x \partial_x -y\partial_y-2 z \partial_z,y\partial_x-z^2 \partial_z, x \partial_x+y\partial_y, \\
\qquad x^2 \partial_x+xy\partial_y+(y-xz)\partial_z, xy\partial_x+y^2 \partial_y+z(y-xz) \partial_z \rangle, \\
\langle \partial_x ,\partial_y ,x \partial_y ,x \partial_x -y \partial_y ,y \partial_x ,x \partial_x +y \partial_y +\partial_z,x^2 \partial_x +x y \partial_y +\tfrac{3}{2} x \partial_z,x y \partial_x +y^2 \partial_y +\tfrac{3}{2} y \partial_z \rangle, \\
\langle \partial_x, \partial_y, x\partial_y, x\partial_x-y\partial_y , y\partial_x, x \partial_x +y\partial_y, \partial_z, x^2 \partial_x +xy\partial_y +x\partial_z, xy\partial_x+y^2 \partial_y +y\partial_z \rangle.
  \end{gather*}
Here, we have changed coordinates for the last Lie algebra [Lie{\hskip1pt}31] in order to avoid the singular orbit $z=0$, and after this coordinate change it is clear that it contains the fourth one [Lie{\hskip1pt}29]. The coordinate transformation $z\mapsto z-xy/2$ takes the first Lie algebra [Lie{\hskip1pt}16] to
\[ \langle \partial_x-\tfrac{1}{2} y \partial_z, \partial_y+\tfrac{1}{2} x \partial_z, x \partial_y,x\partial_x-y\partial_y,y\partial_x, \partial_z\rangle\]
which is obviously a Lie subalgebra of the second one [Lie{\hskip1pt}17]$_{h=1}$.

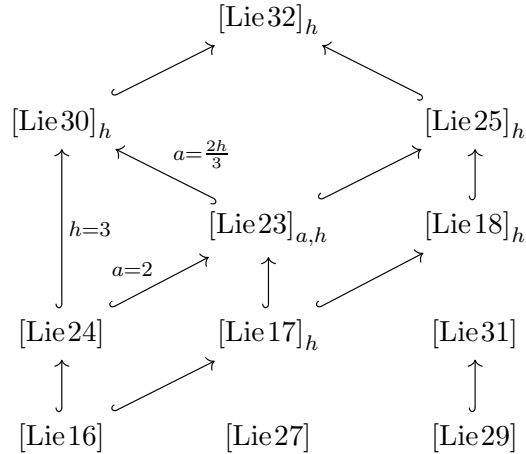
\begin{figure}[h]
\begin{tikzcd}
& & \text{[Lie{\hskip1pt}32]}_h \\
&\text{[Lie{\hskip1pt}30]}_h\arrow[hook]{ur}  & & \text{[Lie{\hskip1pt}25]}_h \arrow[hook]{ul}    \\
& &\text{[Lie{\hskip1pt}23]}_{a,h}\arrow[hook]{ur}\arrow[hook]{ul}[swap]{a=\frac{2h}{3}} &\text{[Lie{\hskip1pt}18]}_h\arrow[hook]{u}        \\
&\text{[Lie{\hskip1pt}24]}\arrow[hook]{uu}[swap]{h=3}\arrow[hook]{ur}{a=2}  & \text{[Lie{\hskip1pt}17]}_h\arrow[hook]{u} \arrow[hook]{ur} &\text{[Lie{\hskip1pt}31]}    \\
&\text{[Lie{\hskip1pt}16]}\arrow[hook]{u} \arrow[hook]{ur} &\text{[Lie{\hskip1pt}27]} & \text{[Lie{\hskip1pt}29]} \arrow[hook]{u}
\end{tikzcd}
\caption{Diagram of inclusions for the Lie algebras listed in Proposition \ref{prop:imprimitive}. All inclusions hold for any fixed $h \geq 1$, except for the one that is marked.}
\label{fig:imprimitive}
\end{figure}

Thus, every Lie algebra that preserves a 1-dimensional foliation and no 2-dimensional foliation contains a Lie subalgebra which is locally equivalent to [Lie{\hskip1pt}16], [Lie{\hskip1pt}27] or [Lie{\hskip1pt}29], corresponding to \eqref{L6}, \eqref{L7} and \eqref{L8}, respectively.
Notice that  [Lie{\hskip1pt}16] (abstractly $\mathfrak{sl}(2) \ltimes \mathfrak{heis}(3)$) does not embed into [Lie{\hskip1pt}27] or [Lie{\hskip1pt}29] (abstractly $\mathfrak{sl}(3)$). Moreover, [Lie{\hskip1pt}27] and [Lie{\hskip1pt}29] are clearly different realizations of $\mathfrak{sl}(3)$ since [Lie{\hskip1pt}27] preserves a contact distribution, while [Lie{\hskip1pt}29] does not.
Figure \ref{fig:imprimitive} shows a more comprehensive tree of inclusions of these Lie algebras. We note also that \mbox{[Lie{\hskip1pt}16]} embeds into the realization of $\mathfrak{sp}(4)$ that we considered in \S\ref{sect:Sytems}.
However, since the last one is primitive while the first one is not, we separated them in our treatment.

\begin{remark}
To finish the proof of Proposition \ref{prop:imprimitive}, one must also verify that none of the Lie algebras listed preserve a 2-dimensional foliation. Since they all contain either [Lie{\hskip1pt}16], [Lie{\hskip1pt}27] or [Lie{\hskip1pt}29], it is sufficient to verify that these three do not preserve a 2-dimensional foliation. We skip the details, but note
that the Lie algebras [Lie{\hskip1pt}16] and [Lie{\hskip1pt}27] both preserve a contact structure, given by the contact forms $dz-y dx$ and $dy-zdx$, respectively. The Lie algebra \mbox{[Lie{\hskip1pt}29]} preserve no 2-dimensional distribution.
\end{remark}

\bigskip

\textsc{Acknowledgment.} The question starting this note arose in a discussion with Dennis The, whom we kindly thank.
We are also grateful to helpful remarks by Boris Doubrov, Niky Kamran and Vladimir Sokolov
on the first version of the manuscript. 

The research leading to our results has received funding from the Norwegian Financial Mechanism 2014-2021
(project registration number 2019/34/H/ST1/00636) and the Tromsø Research Foundation
(project “Pure Mathematics in Norway”). A considerable part of the work on this project was done while E. Schneider was employed by Faculty of Science at the University of Hradec Králové
and was receiving full support from the Czech Science Foundation (GA\v{C}R no. 19-14466Y).




\end{document}